\documentclass[11pt]{amsart}
\usepackage{amsmath,amssymb, amsthm}
\usepackage{pdfsync}

\newcommand\Z{{\mathbb Z}}
\newcommand\G{{\Gamma}}
\newcommand\g{{\gamma}}
\newcommand\R{{\mathbb R}}
\newcommand\Q{{\mathbb Q}}

\newcommand\hyp{{\mathbb H}}
\newcommand\glnr{{GL_n(\R)}}
\newcommand\glnq{{GL_n(\Q)}}
\newcommand\glnz{{GL_n(\Z)}}

\newcommand\slnz{{SL_n(\Z)}}
\newcommand\semi{\rtimes}

\newcommand\U{{\mathcal U}}
\newcommand\F{{\mathcal F}}
\newcommand\QI{{\bar{QI}}}
\newcommand\inv{{^{-1}}}
\newcommand\by{{\rtimes}}
\DeclareMathOperator\h{hol}
\newtheorem{theorem}{Theorem}[section]
\newtheorem{proposition}[theorem]{Proposition}
\newtheorem{lemma}[theorem]{Lemma}
\newtheorem{claim}[theorem]{Claim}
\newtheorem{lemmadef}[theorem]{Lemma/Definition}
\newtheorem{corollary}[theorem]{Corollary}
\newtheorem{defn}[theorem]{Definition}

\newtheorem{question}[theorem]{Question}
\newtheorem{conjecture}[theorem]{Conjecture}

\newtheorem{remark}[theorem]{Remark}
\begin{document}
\title{Coarse Bundles}
\author{ Kevin Whyte}
\thanks{ DMS-0204576 }

\begin{abstract}  We develop a coarse notion of bundle and use it to understand the coarse geometry of group extensions and, more generally, groups acting on proper metric spaces.   The results are particularly sharp for groups acting on (locally finite) trees with Abelian stabilizers, which we are able to classify completely.    
\end{abstract}

\maketitle

\section*{Introduction}

Our original motivation for the discussion that follows is the quasi-isometric geometry of finitely generated groups.  Specifically,  groups with cocompact isometric actions on spaces which are, at least in the coarse sense, locally compact.   Particularly groups acting on finite valence trees ( the groups are then "homogeneous graphs of groups" in  the sense of \cite{MSW1}, see section \ref{sectreebase}) and groups which are extensions of groups with nice pieces - $\Z^n \rtimes \slnz$ or $\hbox{Aut}(\Sigma_g)$ for $\Sigma_g$ a higher genus surface (sections \ref{section:proper} and \ref{sec:pattern}).

Our point of view on these topics is fundamentally dynamical.   Consider a group $\G = H \rtimes G$, for some map $G \to \hbox{Aut}(H)$.   The geometry of $\G$ encodes the "large scale dynamics" of the $G$ action on $H$.   The simplest example is $\G_{\phi} = H \rtimes _{\phi} \Z$, where $\phi$ is an automorphism of $H$.   We will consider $\phi_1$ and $\phi_2$ to have the same large scale dynamics on $H$ if, for some $(K,C)$, there is a sequence $\{f_n\}_{n \in \Z}$ of $(K,C)$-quasi-isometries of $H$ and an $R>0$ so that for all $n$:  $$d(f_n \circ \phi_1, \phi_2 \circ f_{n+1}) \leq R$$

This says that the iterates of  $\phi_1$ shadow those of $\phi_2$ to within bounded distortion.  This sense or dynamics measures only the ways in which the metric is distorted, and, in particular, considers any two isometric actions to be equivalent.    This is a very natural subject in its own right; historically, the first examples are a closely related notion of coarse dynamics for pseudo-Anosov actions on surfaces is studied in \cite{Schwartz} and the classification of linear $\Z$ actions on $\R^n$ (\cite{FMabc}).    However, the subject is closely related to many earlier results, for example, the results of \cite{benerdete} and \cite{Witte} on the rigidity of affine foliations.  The relationship of this to the quasi-isometric geometry of groups can be seen here as the observation that the dynamics of $\phi_1$ and $\phi_2$ are coarsely equivalent if and only if there is a quasi-isometry between $H \rtimes _{\phi_1} \Z$ and $H \rtimes _{\phi_2} \Z$ which commutes, up to bounded distance, with the projections to $\Z$.  Part of the power of these methods comes from the fact that, in many circumstances, any quasi-isometry respects the dynamics (see section \ref{sec:pattern}).

We describe such groups as coarse versions of bundles.   We give a general cohomological classification which parallels the classification of bundles via group cohomology.    The classification of extensions with kernel $K$ is quite different depending on whether $K$ has center - if it does not, the map to $Out(K)$ is essentially a complete invariant, whereas in the Abelian case the classification is via $H^2$ of the quotient.    We get a decent analogue of the former (which we call the tame case) but in the latter case we can only get results under some fairly strong assumptions.   

Some of the major results of the paper are as follows:

\begin{theorem} \label{thm:graphofzn} Let $G$ be the fundamental group of a finite graph of groups with all vertex and edge groups commensurable to $\Z^n$, and let $T$ be the Bass-Serre tree.  Assume $T$ has infinitely many ends.  Let $h: T \to GL_n(\Q)$ the natural modular homorphism.   We say two such groups have Hausdorff equivalent holonomy if the images of $h$ in $GL_n(\R)$ can be conjugated to be at finite Hausdorff distance.     

\begin{enumerate}
\item If two such groups are quasi-isometric then they have Hausdorff equivalent holonomy.
\item Groups within a given equivalence class of holonomy divide into three quasi-isometry invariant subclasses:
\begin{itemize}
\item Those which are virtually of the form $\Z^n \by F$ for $F$ a free subgroup of $GL_n(\Z)$.  
\item Those which are virtually ascending HNN extensions of some endomorphism $E: \Z^n \to \Z^n$.    These are classified up to quasi-isometry in \cite{FMabc}.
\item All groups not of the first two forms, all of which are in a single quasi-isometry type.  
\end{itemize}
\end{enumerate}
\end{theorem}

In general the commensurability classification in the third case is complicated.  Also, the structure of subgroups of $GL_n(\R)$ at bounded Hausdorff distance from a given subgroup seems to be quite subtle in general (see section \ref{sec:HausdorffGLnR}) - we are able to give a complete answer only for $n=1$ or $2$.    One case where both of these issues are easier is semi-direct products:

\begin{corollary} For semi-direct products $\Z^n \by_\alpha F$ for $F$ free and non-Abelian the image of $F$ up to Hausdorff equivalence is a quasi-isometry invariant.  Further, there are at most two quasi-isometry classes for a  given (Hausdorff class of) image - those for which $F$ is injective and those of which it is not.
\end{corollary}

The trichotomy in theorem \ref{thm:graphofzn} is more general than for graphs of $\Z^n$, applying to any bundles over trees with coarsely locally compact structure groups (see theorem \ref{thm:trichotomy}).  

When the base of the bundle is of higher dimension the issues are more complicated.  If one is thinking by analogy with extensions, the one dimension case avoids all the $H^2$ contributions.    We have some results here as well, but not as strong:

\begin{theorem} \label{thm:central} 
Two central extensions $G_1$ and $G_2$ of $Q$ by $\Z^n$ are quasi-isometric over $Q$ iff  the cocycles defining the extensions are in the same $\glnr$ orbit inside  ${H^2}_\infty(Q,\R^n)$. 
\end{theorem}    

If one knows that two fibers of a bundle sufficiently far apart must be highly distorted (for example if the total space is hyperbolic) then the classification is easier.  See section \ref{section:proper} for details.  One application is :

\begin{theorem}\label{glnzbyzn} Let $\G = \Z^n \rtimes GL_n(\Z)$.  The quasi-isometry group of $\G$ is ${H^1}_\infty(GL_n(\Z), \R^n) \rtimes PGl_n(\Q)$.  Any groups quasi-isometric to $\G$ is commensurable to an extension of $\Z^n$ by a finite index subgroup of $GL_n(\Z)$ which acts in the standard way on $\Z^n$
\end{theorem}

\subsection{Acknowledgments}

These results were discovered in 2003 and conversations with Benson Farb, David Fisher, Lee Mosher, and Michah Sageev were all instrumental.   I would like to especially thank Martin Bridson and Benson Farb, without their frequent encouragements the present paper would likely still be unfinished.    I would also like to thank the NSF for their support throughout the process.

\subsection{Notation}

For us a space will, unless otherwise explicitly stated, mean a coarse path space of bounded geometry.   These are precisely the spaces quasi-isometric to graphs of bounded valence.     We will frequents also need to assume our spaces are {\bf uniformly simply connected}, by which we mean that any such model graph can be made simply connected by the addition of $2$-cells of bounded diameter.   For Cayley graphs of groups these conditions amount to finite generation and presentation respectively.

We use the notation $\QI(X)$ for the group of self-quasi-isometries of a space up to the equivalence relation of bounded distance.    When we refer to a quasi-isometry we mean a specific map, not an equivalence class.  This distinction is usually blurred for convenience of exposition, but it will be important in several points in this work.

\section{Coarse Bundles}

In general one can think of a bundle as either a map from the total space to the base (when the fibers 
are the point inverses or something related) or as a foliation of the total space by the fibers (when the base is the leaf space).   Both points of view will be useful for us.  Throughout the following definitions, we will keep returning to the same classes of examples to see what the definition mean in familiar settings.    We start with the most general sort of fibration:

\begin{defn} A {\bf coarse fibration} is a space $X$ (called the total space) and a collection $\F$ of subsets (called the fibers), satisfying, for some $(K,C)$:

\begin{itemize}
\item For each $x \in X$ there is $A \in \F$ with $d(x,A) \leq C$
\item For all $A$ and $B$ in $\F$, $d^H(A,B) \leq K d(A,B) + C$.  In particular, all the elements of $\F$ are at finite pairwise Hausdorff distances.
\end{itemize}

The set $\F$ equipped with the metric coming from Hausdorff distance in $X$ is called the base of the fibration.
\end{defn}

There is a canonical coarse Lipschitz map from $X$ to the base, sending $x$ to any $A$ in $\F$ as in the first condition.  By the second, any two choices are within Hausdorff distance $C(2K+1)$, and so the map is well defined as a coarse map.  The second of the conditions says precisely that this map is coarse Lipschitz.  One can just as easily define a coarse fibration via this map:

\begin{lemmadef}A coarse Lipschitz map $p: X \to B$ is a coarse fibration if and only if it satisfies the {\bf coarse path lifting property}: there are $(K,C)$ such that for any $b,b' \in B$  and $x \in X$ with $d(p(x),b) \leq C$ there is an $x' \in X$ with $d(p(x'),b')\leq C$ and $d(x,x') \leq K d(b,b') + C$.   \end{lemmadef}

The fibers are then the inverse images of balls of radius $C$. We note that coarse path lifting is equivalent to the same statement with the bound $Kd(b_1,b_2) + C$ replaced with any function $f$ of $d(b,b')$.   To see this, fix $R>0$ large enough so that $B$ is a $R$-coarse path space.   Choose an $R$-path $b=b_1,b_2, \cdots, b_n=b'$.   Sequentially lift the path: find $x_1$ with $d(p(x_1),b_1) \leq C$ and $d(x,x_1) \leq f(R)$, then find $x_2$ of $b_2$ with $d(p(x_2),b_2) \leq C$ and $d(x_2,x_2) \leq f(R)$, \ldots.  This gives an $f(R)$-path in $X$ beginning at $x$ and ending at an $x'$ with $d(p(x'),b') \leq C$.  The length of this path is at most $f(R)n$ and $n$ is bounded linearly in $d(b,b')$, so the path lifting property holds.  This construction is the motivation for the terminology.

\begin{proof}

If $p:X \to B$ has the path lifting property, we define the fibers as the inverse images under $p$ of balls of radius $C$ in $B$, $\F = \{ A_b: p^{-1}(B(b,C)\}_{b\in B}$.  Since $p$ is $C$-onto, these fibers cover $X$.    Consider two fibers $A_b$ and $A_{b'}$.   If $x \in A_b$ and $x' \in A_{b'}$ then $d(p(x),p(x')) \geq d(b,b') - 2C$. Since $p$ is coarse lipschitz, $d(x,x')$ is bounded below by an affine function of $d(p(x),p(x'))$.  Thus $d(A_b,A_{b'})$ is bounded below by an affine function of $d(b,b')$.  So we need to check that $d^H(A_b,A_b')$ is bounded above by an affine function of $d(b,b')$.   This is precisely the statement of the coarse path lifting property.

Conversely, suppose we have a space $X$ and a collection $\F$ which define a coarse fibration.  To avoid confusion, let $K_0$ and $C_0$ be the constants in the definition of coarse fibration.  As above define a coarse lipschitz map $p:X \to B$ where $B$ is the set $\F$ with the Hausdorff metric.  We need to check that $p$ has the coarse path lifting property.   Let $b$ and $b'$ be points of $B$, and let $x \in X$ have $d_B(p(x),b) \leq C_0$.   By the definition of $p$ there is a $y \in p(x)$ with $d(y,x) \leq C_0$.  Since $d_B(p(x),b) = d^H(p(x),b)$ there is a $z$ in $b$ with $d(z, y) \leq C_0$, and so $d(x,z) \leq 2C_0$.   Now, since $d_B(b,b')$ is the Hausdorff distance from $b$ to $b'$ there is an $x'$ in $b'$ with $d(x',z) \leq d(b,b')$.    Since $d(x',p(x')) \leq C_0$, we have $d_X(p(x'),b') \leq C_0$ and by the definition of a coarse fibration, $d^H(p(x),b') \leq (K_0+1)C_0$.  So, $d_B(p(x'),C_0) \leq (K_0+1)C_0$ and $d(x,x') \leq d(b,b') + 2C_0$.  Thus $p$ is a coarse fibration with $K=1$ and $C=\max(2C_0, (K_0+1)C_0)$.

\end{proof}

{\bf Example: coset foliations}
Let $\G$ be a finitely generated group, and let $A$ be a subgroup.  One has from this
two natural foliations of $\G$ : the left and right coset foliations.   Dually, one has the two quotient maps from $\G$ to $A \backslash \G$ and $\G / A$.  

Consider first the foliation by right cosets $Ag$ (or, equivalently, the map $\G \to A \backslash \G$).   Given cosets $Ax$ and $Ay$, the distance between $ax$ and $ay$ is independent of $a$, and so the cosets are at finite Hausdorff distance.   Further, this distance is equal to the infimum of distances of coset representatives, so $d^H(Ax,Ay) = d(Ax,Ay)$.   Thus one gets a coarse fibration.  The base of this fibration, $A \backslash \G$, can be thought of as the graph with vertices $A \backslash \G$ and edges between $Ax$ and $Axs$ for all $s$ in the generating set for $\G$.  Notice that this is not the normal construction of the coset graph, and in particular, $\G$ does not act on this space (of coarse, $\G$ does act on $A \backslash \G$ by right multiplication, but this does not respect the metric structure).  

Next, consider the foliation of $\G$ by left cosets.   This will typically not be a coarse fibration.   A coset $xA$ is at finite Hausdorff distance from $A$ if and only if  $A$ and $xAx^{-1}$ are commensurable.   So to have a coarse fibration one must have that every conjugate of $A$ is commensurable to $A$.  We call such a subgroup {\bf almost normal}.  If $A$ is normal then the foliations by left and right cosets are the same, so we are in the previous case (note that this is precisely the subset of the previous case where there is a natural isometric $\G$ action).   For any subgroup $A$, the space $\G/A$ has a natural graph structure,  the coset graph of $A$, where $xA$ and $yA$ are adjacent if $x^{-1}y \in ASA$ where $S$ is the given generating set for $\G$.   Note that this graph has finite valence precisely when $A$ is almost normal.  

Coset foliations can be generalized somewhat to include all the examples of interest in this paper:

\begin{lemma}\label{lemma:action} Let $\G$ be a finitely generated group, and let $X$ be a metric space of coarsely bounded geometry.    Suppose $\G$ acts isometrically on $X$ with bounded quotient.  For any $x_0 \in X$ the map $\G \to X$ sending $\g \mapsto \g x_0$ is a coarse fibration. 
\end{lemma}

\begin{proof}

Fix $x_0$ and let $p$ be the map $p(\g) = \g x_0$.   We have $d(p(\g),p(\g'))= d(\g x_0, \g' x_0) = d(x_0, \g^{-1} \g' x_0)$.   If $\g^{-1}\g' = s_1 s_2 \cdots s_n$ then $d(x_0, \g^{-1} \g' x_0) = d(x_0, s_1 \cdots s_n x_0) \leq \Sigma_{i=1}^n d(s_1 \cdots s_{i-1} x_0, s_1 \cdots s_i x_0) \leq n \sup_i d(s_i x_0, x_0)$.  So $p$ is lipschitz with constant the maximum distance a generator moves $x_0$. 

As remarked in the definition of coarse path lifting, it suffices to check that there is a $C>0$ such that for all $R$ there exists an $R'$ so that given $\g$ in $\G$ and $x,x' \in X$ with $d(x,x') \leq R$ and $d(\g x_0, x) \leq C$ that there is a $\g' \in \G$ with $d(\g' x_0, x') \leq C$ and $d(\g, \g') \leq R'$.   Since everything is $\G$ invariant, we may assume $\g$ is the identity.   Path lifting then amounts to the following:  there is a $C>0$ such that for any $R$ there is $R'$ so that if $d(x,x_0) \leq R$ there is a $\g$ with $|\g| \leq R'$ and $d(x,\g x_0) \leq C$.  

Since $X$ has bounded geometry there is an $r_0$ such that for any $R$, $B(x_0,R)$ is covered by finitely many balls of radius $r_0$.   Since the $\G$ action is cobounded, there is an  $r_1$ (still independent of $R$) so that a finite collection of $r_1$-balls centered at points in $\G x_0$ cover $B(x_0,R)$.   Let $A \subset \G$ be a finite set such that the $r_1$ balls centered at $A x_0$ cover $B(x_0,R)$.   Then we have the desired property with $C=r_1$ and $R' = \sup_{\g \in A} |\g|$.

\end{proof}

If $X$ is a locally finite simplicial complex and the $\G$ action is simplicial, then $X$ is quasi-isometric to $\G/A$ where $A=stab(x_0)$ (note that the commensurability class of $A$ does not depend on $x_0$, and that $A$ is almost normal), and the coarse fibration $\G \to X$ reduces to the left coset foliation.   If the action of $\G$ on $X$ is not simplicial, for example $\Z^2$ acting on $\R$ by two rationally independent translations, then the fibration $\G \to X$ need not be equivalent to any coset foliation.  On the other hand, it is possible for $\G/A$ to have coarsely bounded geometry even when $A$ is not almost normal (indeed, if $\G$ is boundedly generated by conjugates of $A$ then $\G/A$ is bounded).   In these cases the map $\G \to \G/A$ is a coarse fibration, but the fibers are not coarsely equivalent to the cosets of $A$.

We will be interested in somewhat more restricted types of coarse fibrations.  In particular, in a coarse fibration the fibers are all quasi-isometric (in the metrics induced from $X$) but not uniformly so (consider the fibers in a right coset foliation).    When, as in the examples here, there is a cocompact groups of symmetries one naturally has such identifications.   When we have a fixed model space for the fiber we call it  a {\bf coarse bundle}:

\begin{defn} Fix a space $F$.  A coarse fibration $p:X \to B$ is a {\bf coarse bundle} with fiber $F$ if there is an $r>0$,  a proper function $\rho:\R^+ \to \R^+$ and $(K,C)$  such that every fiber of $p$ is a within Hausdorff distance $r$ of a $(K,C,\rho)$ uniformly proper embedding of $F$ in $X$.
\end{defn}

Since all the fibers of any coarse fibration are quasi-isometric, one could simply demand that all the fibers are uniformly quasi-isometric to $F$.   The point of the definition is that we generally think of the model fiber $F$ in an intrinsic path metric which is usually different from the induced metric as a subset of $X$.  It would be equivalent to define a coarse bundle with fiber $F$ to be a coarse fibration such that for some $r>0$ the $r$-path metrics on fibers are uniformly quasi-isometric to $F$.   Note that (as will always be the case for us), if $F$ is a coarse path space then any coarse bundle has fibers which are uniformly coarsely connected in $X$. 

{\bf Example: cosets foliations again} Let $\G$ be a finitely generated group and $A$ a finitely generated subgroup.  For the coarse fibration $\G \to A \backslash \G$ to be a coarse bundle, all the right cosets of  $A$ must be (uniformly) coarsely connected.   Since left translations are isometries, this is equivalent to assuming that all the conjugates of $A$ are $r$-coarsely connected for some $r$.   This means they are each generated by their intersections with the ball of radius $r$.   There are only finitely many subsets of the ball of radius $r$, so this implies there are only finitely many conjugates of $A$ - so $A$ is normalized by a finite index subgroup of $\G$.   Thus essentially the only right coset fibrations that are coarse bundles are those that are equivalent to left coset fibrations.

On the other hand coarse fibrations coming from lemma \ref{lemma:action} have all fibers uniformly quasi-isometric since these fibrations are invariant under the isometric actions of $\G$.    Thus they are coarse bundles as soon as any one fiber is coarsely connected.   For simplicial actions this amounts to assuming the simplex stabilizers are finitely generated.  For more general actions this is related to Bieri-Neumann-Strebel invariants, see \cite{Geo} for some interesting work on this for actions on $CAT(0)$-spaces.

\begin{lemmadef} Let $X \to B$ be a coarse bundle with fiber $F$.  For any $r>0$ there are $(K_r,C_r)$ so that for any two points $b$ and $b'$ with $d(b,b') \leq r$ the closest point projection map  between the fibers of $X$ over $b$ and $b'$ induces a $(K_r,C_r)$ quasi-isometry, $\phi_{b,b'}$ between the fibers in their intrinsic metrics.   These quasi-isometries are well-defined to within a distance $D_r$.  Further, if $\psi$ is any map between the fiber over $b$ and the fiber over $b'$ such that the distances in $X$ bewteen $f$ and $\psi(f)$ are uniformly bounded for $f$ in the fiber over $b$ then $\psi$ is at bounded distance from $\phi_{b,b'}$.

\end{lemmadef}

\begin{proof}

By assumption there is an $R$ (depending only on $r$) so that the given fibers are at Hausdorff distance at most $R$.  Thus, closest point projection, as a map between the fibers with their induced metrics, is a $(1,2R)$ quasi-isometry.   Since the fibers have uniform distortion, and are coarse path spaces the desired $(K_r,C_r)$ exist.  While the closest points need not be unique, any two closest points are within $2R$ as measured in $X$.   This means the quasi-isometries are within $D_r=\rho^{-1}(2R)$ measured in the fibers.  Likewise, if $\psi$ is a map from the fiber over $b$ to the fiber over $b'$ with $d(f,\psi(f))\leq M$ for all $f$, then the distance from $\psi(f)$ to any closest point to $f$ in the fiber over $b'$ is at most $M+R$.  Thus the distance between $\psi$ and $\phi_{b,b'}$ is at most $\rho^{-1}{R+M}$.

\end{proof}

This gives an important way of viewing coarse bundles - as built out of local gluing data on the base.    
Given a coarse bundle over $B$ with fiber $F$, choose, for each fiber $b \in B$ a $(K,C)$ quasi-isometry between the fiber over $b$ (in its path metric) and $F$.   Using these identifications of the fibers with $F$ and the previous lemma, for every pair $b$ and $b'$ we have a quasi-isometry of $F$ whose constants depend only on $d(b,b')$.  In many cases we can see this connecting data fairly explicitly:

{\bf Example: group extensions:} Let $\G$ be a finitely generated group and $A$ a normal subgroup.   For each left coset $x \in \G/A$ choose a coset representative $\hat{x}$, and identify the fiber over $x$ with $A$ by the map $a \mapsto \hat{x} a$.  Since left translations are isometries, these are all uniform quasi-isometric identifications.    Now, given $x$ and $y$ in $\G/A$,  let $\delta$ be of minimal norm in $x^{-1}yA$.   Right translation by $\delta$ move each point of the fiber over $x$ to a closest point in the fiber over $y$ (since left and right cosets are equal).   Under the given identifications with $A$, the map $\phi_{x,y}$ is conjugation by $\hat{y}^{-1} \hat{x}$ composed with right translation by $\hat{y}^{-1} \hat{x} \delta$ (note that this last term in in $A$).   If $\G$ is the semi-direct product $A \semi \G/A$ and we choose the elements of $\G/A$ as our coset representatives then the map $\phi_{x,y}$ is just conjugation by $x^{-1}y$ on $A$.

Under mild assumptions the gluing data determines the bundle: Fix a bundle $X$ with associated data $\phi_{b,b'}$ as above.   From this one can build the space $F \times B$ with the maximal metric such that distances along the fibers $F \times b$ are at most the distance in $F$ and the distance between $(f,b)$ and $(\phi_{b,b'}(f),b')$ are bounded by $d(b,b')$.   By construction the map from this space to $X$ is coarse Lipschitz, and it is not hard to see it is a quasi-isometry (this is where the assumption that the fibers are uniformly properly embedded comes in).

However, not all collections of such maps arise from coarse bundles.  Suppose we are given $B$ and $F$, and for every $b,b'$ a quasi-isometry $\phi_{b,b'}$ whose constants depend only on $d(b,b')$.   From the previous discussion we know how to build a model of what the space $X$ would need to be.  There are compatibility conditions that must be satisfied to get a bundle, namely, for for any $r$-path $b_1,b_2, \cdots, b_n$ the composition $\phi_{b_{n-1},b_n} \phi_{b_{n-1},b_{n-2}} \cdots \phi_{b_1,b_2}$ has to be at bounded distance from $\phi_{b_1,b_n}$, and this distance must be bounded only in terms of $n$ and $r$.    

If $B$ is an $r$-path space, this means that we need only specify $\phi_{b,b'}$ for $d(b,b') \leq r$ to determine the bundle.  These will all be uniform quasi-isometies.  The compatibility condition then says that the composition of these maps around a loop in the $r$-thickening of $B$ is at bounded distance from the identity, with the distance bounded in terms of the length of the loop.    Thus it is natural to try to think of this data as some sort of cohomology class on $B$ with coefficients in quasi-isometries of $F$.   precisely what this means is trickier than it might seem at first.

\begin{defn} Let $Ext(B,F)$ be the set of coarse bundles over $B$ with fiber $F$ up to coarse equivalence over $B$ (meaning quasi-isometries between the total spaces which commute up to bounded distance with the maps to $B$).
\end{defn}

We must define what we mean by cohomology with coefficients in $\QI(F)$.    For our purposes what is relevant is a version of $L^\infty$-cohomolgy.  We give a brief discussion here and point the reader to \cite{CoarseHomology} for a more in depth discussion of the issues around coarse cohomology theories.

\begin{defn} Let $X$ be a simplicial complex and $F$ a coarse path space.  Define ${C^i}_\infty (X, QI(F))$ as the set of maps from the (oriented) $i$-cells of $X$ to $\QI(F)$ for which the image is uniform (has uniform quasi-isometry constants).    
\end{defn}

Since $\QI(F)$ is not Abelian, we cannot define a coboundary in general, but as usual can do so well enough to define $H^1$.    We call an element $c$ of $C^1$ closed if for every triangle $\sigma$ with edges $e_1,e_2,e_3$ we have $c(e_1)c(e_2)c(e_3) = Id$.   Let $Z^1_\infty$ be the subset of closed elements of $C^1_\infty$.   Two elements $c_1$ and $c_2$ of $C^1$ are cohomologous if there is an $f$ of $C^0$ for which $c_2(e) = f(\tau e)\inv c_1(e) f(\iota e)$ for all edges $e$.  We define ${H^1}_\infty(X,\QI(F))$ as the set cohomology classes of closed elements of ${C^1}_\infty$.

To apply this in a coarse setting we do the usual trick of thickening $B$ to its $r$-Rips complex $R_r(B)$.     We can then look at ${H^1}_\infty (R_r(B),\QI(F))$ for various $r$.   If $B$ is uniformly simply connected then this stabilizes for some $r$ - if not we must take limits (or, better, use the pro-system).   Coarse simple connectivity is a non-trivial restriction - for groups it is equivalent to finite presentability.    However, all the examples in the current paper are coarsely simply connected so for simplicity we will assume that - the more general case is similar but with more cumbersome notation.

\begin{lemmadef}  Fix spaces $B$ and $F$ as above.  There is a well defined map $H:Ext(B,F) \to {H^1}_{\infty} (B,\QI(F))$.  
\end{lemmadef}

\begin{proof}

We have already seen how to define $H(p)$ for a bundle over $B$ with fiber $F$.   Namely, choose uniform quasi-isometries of the fibers with $F$ and then for any pair $(b,b')$ let $\phi_{b,b'}$ be the equivalence class of any closest point projection from the fiber over $b$ to the fiber over $b'$, thought of as a self map of $F$.   For any $r$, the values of this map on pairs with $d(b,b') \leq r$ lie in a uniform subset of $\QI(F)$.    For any coarse $2$-simplex in $B$, $(b,b',b'')$, $d\phi(b,b',b'') = \phi(b'',b)\phi(b',b'')\phi(b,b')$.  This is represented by the composition of the corresponding closest point maps, and so defines a map from $F$ to itself which represents a map from the fiber over $b$ to itself which moves points a bounded amount, and is therefore the trivial class in $\QI(F)$.   This shows that $d\phi = 0$.   We define $H(p)$ as the cohomology class of $\phi$.

Suppose $p':X' \to B$ is another bundle with fiber $F$ which is equivalent to $p$.  This means that there is a quasi-isometry $g:X \to X'$ such that $(p'(g(x))=p(x)$ for all $x$.    Using the identifications of the fibers with $F$, this gives a map $G : B \to QI(F)$ which takes values in a uniform subset of $QI(F)$.  The fact that these fiberwise quasi-isometries assemble to a quasi-isometry $X' \to X$ means that for all pairs $(b,b')$ in $B$, $G(b') [ \phi'_{b,b'}] G(b)^{-1}$  is a map between the fibers in $X$ over $b$ and $b'$ which moves points a uniformly bounded distance in $X$, and so is at bounded distance from $\phi_{b,b'}$ and so equal to it in $\QI(F)$.  This shows that $\phi$ and $\phi'$ are uniformly cohomologous, and so define the same element of ${H^1}_\infty(B,\QI(F))$.

\end{proof}

The cohomology class $H(p)$ does not determine the bundle. The difficulty is that we had to pass to the equivalence classes of quasi-siometries of $F$ to get a cocycle (or, for that matter, an actual group of coefficients).  This means, for example, one gets a trivial cocycle from any bundle where all the maps $\phi_{b,b'}$ are at bounded distance from the identity as quasi-isomtries of $F$.  Two examples of non-trivial bundles of this sort:

Let $\G$ be the $3\times 3$ integral Heisenberg group and let $A$ be the center.    The previous 
calculations show that all the maps $\phi_{b,b'}$ for the bundle $\G \to \G/A = \Z^2$ are translations of the fibers, which are at bounded distance from the identity.   Central extensions in general give examples where the cocycle is trivial.   From the non-coarse point of view this is well understood - there is a secondary invariant in $H^2$ of the base that controls central extensions.  See section \ref{central} for some coarse versions of this.  However, this is not the whole story.   When the base is one dimensional, such issues should not arise.  However:

Let $\phi_n: \Z \to \Z$ be the map defined by $\phi_n(x)=2x$ if $|x| \leq n$, $\phi_n(x) = x + n$ for $x>n$ and $\phi_n(x) = x-n$ if $x<-n$.    It is not hard to check that these maps are all $(2,1)$ quasi-isometries of $\Z$, and are all at finite distance from the identity (although not uniformly so).   Let $X$ be the bundle with these as connecting maps, namely the graph with vertex set $\Z \times \Z$ with $(a,b)$ connected to $(a + 1,b)$, $(\phi_b(a), b+1)$.   This bundle is not equivalent to the trivial bundle $\Z \times \Z$.  Indeed, $X$ is not quasi-isometric to $\Z \times \Z$ as there is no polynomial bound on the volumes of balls (the space $X$ can be thought of as the Euclidean plane modified by removing an infinite wedge and gluing in a piece of the hyperbolic plane in its place - balls in the hyperbolic part grow exponentially).

It is useful to consider further the analogy with the problem of classifying group extensions, say of $B$ by $A$.   As our earlier example explains, the cohomology class $H(p)$ is analogous to the map $B \to Out(F)$.   When $F$ has no center, this is a complete invariant.   When $F$ has center, one has to use a secondary invariant in $H^2(B,Z(F))$ (where $Z(F)$ is the center of $F$ thought of as a $B$-module).   Generalizing this to our setting is problematic.  The analogue of the center is the collection of quasi-isometries at bounded distance from the identity.   This is not Abelian, nor even a group, so defining anything like $H^2$ with these coefficients is prooblematic.   We will come back to this issue later in section \ref{sec:notame}.   First we turn to the simpler case, analogous to the center-free case of group extensions.

\begin{defn}  A space $F$ is {\bf tame} if for every $(K,C)$ there is an $R>0$ so that every $(K,C)$-quasi-isometry of $F$ is either at infinite distance from the identity or moves no point more than $R$.
\end{defn}  

Tameness is discussed in detail in \cite{whyte:tame}, for now we give some basics examples:

\begin{lemma}  If $X$ is a cocompact, non-elementary Gromov hyperbolic metric space then $X$ is tame.
\end{lemma}
 
\begin{proof}

Since $X$ is non-elementary, there are at least three points at infinity.  These three points can
be coned off to get a quasi-isometrically embedded copy of the infinite tripod in $X$.  By cocompactness, any point in $X$ is within a uniformly bounded distance of the vertex of such a
tripod.  Let $f$ be a $(K,C)$ quasi-isometry at finite distance from the identity.  For any tripod $T$ with vertex $x$, $f(T)$ is a quasi-isometrically embedded tripod with the same endpoints.  The three branches of the tripod are quasi-geodesics (with constants depending only on $(K,C)$ and $X$).  By hyperbolicity there is some $r$ so that each is within $r$ of the geodesic ray from $f(x)$ to the endpoint at infinity.  This implies that either $f(x)$ is within $r$ of the corresponding ray of $T$, or $x$ is within $r$ of the ray from $f(x)$.  Applied to all three branches, this implies $f(x)$ is
within $r$ of $x$.  Since every $x$ in $X$ is within a uniform distance, $D$, of the vertex of some tripod, this implies $f$ is within $D+r+KD+C$ of the identity.
 
\end{proof}
 
Tameness is much weaker than hyperbolicity - it holds for all lattices in semi-simple Lie groups with no Euclidean factors, all the Baumslag-Solitar groups, etc.  Indeed the only examples of groups which are not tame known to the author are groups with infinite virtual center.    It would certainly be of interest to know whether this reflects reality - it would show that finite virtual center is a quasi-isometry invariant and might allow to classify nilpotent groups up to quasi-isometry using induction and the central series.

\begin{question} Is tameness for finitely generated (or finitely presented) groups equivalent to finite virtual center?
\end{question}

\begin{remark}We will often use tameness in the following form:  $F$ is tame if for any $(K,C)$ there is an $R$ so that any two $(K,C)$ quasi-isometries which are equivalent are at distance at most $R$. This is essentially immediate from the definition : if $f$ and $g$ are $(K,C)$-quasi-isometries which are equivalent then $f^{-1}g$ is a $(K^2,KC+C)$ quasi-isometry which is at bounded distance from the identity.  Thus, by tameness, it is within $r_0$ of the identity.   Then $d(f(x),g(x)) \leq Kd(x,f^{-1}g(x)) + C \leq Kr_0+C$.  
\end{remark}

\begin{lemma}\label{cohtamespace} If $F$ is tame then for all uniformly simply connected $B$, the map $H:Ext(B,F) \to {H^1}_\infty(B,\QI(F))$ is an isomorphism.
\end{lemma}
 
\begin{proof}

Choose $r>0$ so that $B$ is an $r$-path space.  Let $c$ be an $L^\infty$ $1$-cocycle with coefficients in $\QI(F)$ defined on $\Delta_r(B)$.  Let $(K,C)$ be such that the values of $c$ on $\Delta_r(B)$ take values in $(K,C)$-quasi-isometries.  For any $x$ and $y$ in $B$ with $d(x,y) \leq r$, choose any $(K,C)$-quasi-isometry in the equivalence class of $c(x,y)$ and call this map $\phi_{x,y}$.   We will use these maps to build a coarse bundle as before.  We only need to check that for any $r$-path $b_1, b_2, \cdots b_n$ that the composition  $\psi=\phi_{b_1,b_2} \phi_{b_2,b_3} \cdots \phi_{b_{n-1},b_n}$ depends only on the endpoints, up to an error bounded solely in terms of $n$.   The fact that $c$ is a cocycle says that the equivalence class of  $\psi$ depends only on the endpoints.  Since all the $\phi$  are $(K,C)$-quasi-isometries, $\psi$ is a  $(K^n, \frac{K^n -1}{K-1} C)$-quasi-isometry.   Thus there is an $R$ such that for paths of length at most $n$,  $\psi$ is determined to within $R$ by the endpoints.   Thus this data defines a bundle $p_c: X_c \to B$ with $H(p_c)=[c]$, and so the map $H$ is surjective.

Likewise, if $p: X \to B$ is any bundle with $H(p)=c$ then the element $f \in C^0$ which exhibits the cohomology between the data for $X$ and the cocyle $c$ assembles to a map $X_c \to X$ intertwining $p$ and $p_c$ which is a $(K,C)$ quasi-isometry along fibers and the identity over $B$.    

\end{proof}

\subsection{Structures on bundles} \label{section:structures}
 
Unfortunately, many of the bundles we care about do not have tame fibers.   In particular, all of the bundles coming from actions with Abelian stabilizers via lemma \ref{lemma:action} fail to have tame fibers since $\Z^n$ is not tame.   However, in these examples, we know we can choose the identifications of the fibers with $\Z^n$ so that the connecting maps $\phi_{b,b'}$ are all elements of the abstract commensurator of $\Z^n$, which is $\glnq$.   This extra structure can make the situation easier to analyze.

\begin{defn} A {\bf structure set} on $F$ is a subset $\Sigma \subset QI(F)$ together with a collection of subsets $\U$ of subsets of $\Sigma$ called {\bf uniform subsets}, closed under subsets, satisfying:

\begin{itemize}

\item For every uniform set $U$ there is $(K,C)$ such that $U$ consists of $(K,C)$-quasi-isometries.

\item For every uniform set $U$ there is an $r>0$ and a uniform set $V$ such for any $f$ and $g$ in $U$ there is an $h$ in $V$ with $d(fg,h) \leq r$.

\end{itemize}
 
\end{defn}
 
The full collection of quasi-isometries of $F$ with the usual meaning of uniform is certainly a structure set.   Another example is $\glnr$ as a structure set for $\Z^n$, where uniform means precompact in the usual topology.    A structure set is not quite a group, since the approximate product defined in the second condition need not be unique.   If $\Sigma$ has a group structure such that in the second condition $h$ can be taken to be the product of $f$ and $g$ then we call $\Sigma$ a structure group.   Note that we do not assume that the uniform structure on $\Sigma$ is induced from $QI(F)$, for example the group $Aff(\R^n)$ has two natural notions of uniformity: one which comes from uniformity of quasi-isometry constants acting on $\R^n$ and one in which the uniform sets are the pre-compact sets. 

\begin{defn}  Let $\Sigma \subset QI(F)$ be a structure set.   Given a coarse bundle $p:X \to B$ with fiber $F$, a {\bf $\Sigma$ structure} on $p$ is a choice of $(K,C)$-quasi-isometric identifications of the fibers of $p$ with $F$ (for some $(K,C)$) and for all edges $ e \in B$ an element $f_e$ of $\Sigma$ which is at (uniformly) bounded distance from $\phi_{b,b'}$ such that the collection of compositions of these $f_e$ around the boundaries of $2$-cells of $B$ is a uniform subset of $\Sigma$.
\end{defn}

Essentially, an $(F,\Sigma)$ bundle is a coarse bundle where the connecting maps are chosen from $\Sigma$.   

If $\G$ is a finitely generated group, $A$ a finitely generated normal subgroup, then we have seen that by choosing appropriate identifications of the fibers with $A$ we can arrange for the closest point projection maps to be of the form $ a \mapsto x^{-1} a x b$ for $x \in \G$ and $b \in A$.  Thus these bundles naturally have a $Aut(A) \semi A$ structure.  This can be be reduced to an $Aut(A)$ structure if $\G$ is a semi-direct product.

\begin{lemma}\label{actionstructure} Let $\G$ be a finitely generated group acting with finite quotient on a locally finite simplicial complex $X$.   There is a subgroup $A$ of $\G$ such that for every $\sigma$ a simplex of $X$, $G^\sigma$ is commensurable to $A$.   If $A$ is finitely generated then for any $x_0 \in X$ the orbit map $\G \to X$ sending $\g \mapsto \g x_0$ is a coarse bundle with fiber $A$.  Further, this bundle has a natural $\tilde{Comm}(A)$ structure.
\end{lemma} 

Here $\tilde{Comm}(A)$ is the collection of maps of $A$, defined on cosets of finite index subgroups which are equivariant with respect to some abstract commensurator.

Now, one can define $Ext_\Sigma(B,F)$ as before and try to relate this to some ${H^1}_\infty$ type invariant.    As before one must first pass the a quotient which is a structure group rather than just a set, as follows:

\begin{defn} Let $\Sigma$ be a structure set.  Let $\bar{\Sigma}$ be the image of $\Sigma$ in $\QI(F)$.   Call a subset $U$ of $\bar{\Sigma}$ is uniform if and only if there is a uniform subset $U'$ of $\Sigma$ containing a representative of each equivalence class in $U$.  
\end{defn}

\begin{defn} A structure set is {\bf tame} if for every uniform $U$ there is an $r>0$ so that the elements of $U$ which are trivial in $\bar{\Sigma}$ are within $r$ of the identity.  
\end{defn}
 
 As before the key point here is that for $\Sigma$ tame, $U \subset \bar{\Sigma}$ uniform, any two uniform sets of representatives for $U$ in $\Sigma$ are at uniformly bounded distance.  Exactly as before, we have:

\begin{lemma} There is a map $H:Ext(B,(F,\Sigma)) \to {H^1}_\infty(B,\bar{\Sigma})$.  If $\Sigma$ is tame then this map is a bijection.
\end{lemma}

\subsection{Holonomy}\label{section:holonomy}

The classification of group extensions with fiber $A$, when $A$ has no center, is usually phrased in terms of a map to $Out(A)$ rather than a $1$-dimensional cohomology class.   From this point of view the extensions are simply pull-backs of the universal extension:
 $$1 \to A \to Aut(A) \to Out(A) \to 1$$

It is useful to give a similar rephrasing of the above classification result for bundles with tame structure sets.   

\begin{lemmadef} Let $p:X \to B$ be an $(F, \Sigma)$ bundle.  If $f : C \to B$ is a coarse lipschitz map then there is a canonical pull-back $f^*p:f^* X \to C$ which is an $(F, \Sigma)$ bundle.
\end{lemmadef}
  
Simply define $f^*X$ to be the subset of $X \times C$ of pairs $(x,c)$ with $f(c)=p(x)$ (more precisely, fix $r$ large enough so that $p$ is $r$-onto and look at pairs with $d(f(c),p(x)) \leq r$).   The map $f^*p : f^*X \to C$ is simply projection onto the second coordinate.   It is easy to verify that this defines an $(F,\Sigma)$ bundle and projection to the $X$ coordinate gives a fiberwise isomorphism covering $f$.

For any $(F,\Sigma)$ bundle over $B$ we can "integrate" the cohomology class we constructed above to get a map, called the {\bf holonomy} from $B$ to $\bar{\Sigma}$.  Fix a base point $b_0 \in B$ and define $\h:B \to \bar{\Sigma}$ by defining $\h(b)$ as the equivalence class of the closest point projection from $b$ to $b_0$.   For any $b$ and $b'$ in $B$ we have $\h(b') = \h(b) \pi(b,b')$, where $\pi(b,b')$ is the equivalence class in $\bar{\Sigma}$ of the projection from the fiber over $b'$ to the fiber over $b$.    By definition this is uniform in $d(b,b')$.   Thus, although $\h$ is not uniform, it is "lipschitz" in the obvious sense.  This is just another way of saying that $d \h = H(p)$ is a uniform $1$-cochain.    We can make this Lipschitz in the standard sense as follows:

\begin{defn}
For $U$ a uniform subset of $\bar{\Sigma}$ define $\hat{\Sigma}_U$ as the subgroup of $\bar{\Sigma}$ generated by $U$ equipped with the word metric from this generating set.  
\end{defn}

The holonomy is clearly a Lipschitz map to $\hat{\Sigma}_U$ for $U$ large enough.  As with the ${H^1}_\infty$ invariant, one should really work with an invariant $\hat{\Sigma}$ which is a limit over the net of uniform sets in $\Sigma$.       This is messy and there are not any examples I know of where this is useful as in all the interesting cases where one can say anything this space stabilizes.

\begin{defn} A structure group is {\bf stable} if for all $U \subset V$ uniform sets, the map $\hat{\Sigma}_U$ to $\hat{\Sigma}_V$ is uniformly proper.    The structure group is {\bf strongly stable} if there is $U$ uniform such that for all $V$ uniform with $U \subset V$  the map $\hat{\Sigma}_U$ to $\hat{\Sigma}_V$ is a coarse equivalence.    
\end{defn}

Stability makes sure that the notion of Lipschitz maps to $\hat{Sigma}_U$ (and of bounded distance between maps) does not change if $U$ is elarged. Strong stability allows us to interpret these maps as Lipschitz maps into a metric space.  We will come to some examples shortly, but first we record the upshot of the discussion:

\begin{lemma}\label{holonomy}  Assume $\Sigma$ is stable, then for any $(F,\Sigma)$ bundle over $B$ the holonomy is a lipschitz uniform map $\h: B \to \hat{\Sigma}$ which is well defined up to bounded distance and translation.   If $\Sigma$ is tame then there is a universal $(F, \Sigma)$ bundle $E\Sigma$ over $\hat{\Sigma}$ such that any $(F,\Sigma)$ bundle $p:X \to B$ is equivalent to $\h(p)^* E\Sigma$.   
\end{lemma}

{\bf Examples:} \label{examples:structures}

\begin{itemize}

\item {\bf Euclidan fiber}   The main use we make of this is for Euclidean fibers.   There one has several natural structure groups : translations, linear maps, affine maps, etc.     The linear structure group is tame - no non-trivial linear map is at bounded distance from the identity.    The others are not tame, although in these cases the subgroup $\Sigma_0$ is Abelian, which allows one to make some progress - see section \ref{sec:notame}.    All these groups are strongly stable, with $\hat{\Sigma}$ being the corresponding Lie group with compactly generated word metric.

\item {\bf  Rigid spaces}  We say $F$ is rigid if there is some $(K,C)$ so that all quasi-isometries of $F$ are near $(K,C)$ quasi-isometries.    Examples include higher rank symmetric space and Euclidean buildings (\cite{KL},\cite{EF}), some hyperbolic buildings (\cite{hyperbolicbuilding}, and many other spaces.   In this case the entire set of quasi-isometries is uniform, so $\hat{\Sigma}$ is essentially point.   All bundles with such fibers are quasi-isometric to products (which one can easily prove directly).

\item {\bf Non-uniform lattices}  A more interesting example is a non-uniform lattice is a symmetric space (not $\hyp^2$).    Then one knows that all the quasi-isometries are represented by isometries of the ambient symmetric space which are commensurators of the lattice (\cite{Eskin}, \cite{Schwartz:nonuniform}).    The easiest example to think of here is $GL_n(\Z)$ where the group of quasi-isometries is $GL_n(\Q)$.    While these are all isometries of $GL_n(\R)$, their quasi-isometry constants as maps of $GL_n(\Z)$ depend on how far the lattice is moved off of itself, which translates to the sizes of the denomenators of the matrix entries.     Any uniform subset $U$ is contained in $GL_n(\Z[\frac{1}{S}])$ for some $S$.    The corresponding $\hat{\Sigma}_U$ is coarsely just the product of the buildings for all primes in $S$.    Thus this case is stable but not strongly stable.

\end{itemize}

All of these examples are fairly well behaved, and the spaces $\hat{\Sigma}$ (or at least the $\hat{\Sigma}_U$) are coarsely modeled on locally compact spaces.   This can be formulated coarsely, and we will need it later: 

\begin{defn}\label{clc} A structure group $\hat{G}$ is {\bf coarse locally compact} if there is a uniform set $U_0$ so that for any uniform $U$, there are a finite collection of elements $g_1,\cdots,g_n$ with $U \subset g_1 U_0 \cup \cdots \cup g_n U_0$
\end{defn}

Coarse local compactness implies stability.  See \cite{whyte:tame} for a more thorough discussion of the implications of  coarse locally compactness for subgroups of $\QI(F)$.  In general it is unclear when any of the properties hold.

\begin{question}  For which spaces $F$ is $\QI(F)$ stable? coarse locally compact?  When does $\QI(F)$ have a uniform generating set?  
\end{question}

Holonomy provides a basic invariant of bundles - namely the triple $(B, \hat{\Sigma}, H)$.   If two bundles are equivalent then after translation in $\hat{\Sigma}$ there is a quasi-isometry $B_1 \to B_2$ which commutes up to bounded distance  with the holonomy maps.     We call such a map a {\bf quasi-isometry over $\Sigma$}.     If such a quasi-isometry exists then under fairly weak assumptions Lemma \label{holonomy} gives an equivalence of bundles.

It is often convenient to look for such a quasi-isometry in stages: first one checks that the images $H_1(B_1)$ and $H_2(B_2)$ can be translated to lie at bounded distance and second one tries to build a quasi-isometry over this correspondence.    We call the first property having {\bf Hausdorff equivalent holonomy}, and think of this as recording the type of dynamics that occur in the bundle.   The second problem we call the {\bf lifting problem}.    In the next section we discuss one case where it is trivial, and in section \ref{sec:treebase} we solve it when the base is a tree.

\subsection{Proper holonomy}\label{section:proper}

For a given bundle $X$ with tame fiber, the image of the holonomy map is naturally the smallest possible structure group.   One important case is what we call {\bf proper holonomy}, meaning that the holonomy map is a proper embedding of $B$ in $\hat{\Sigma}$.    It follows immediately from lemma \ref{holonomy} that this property is a coarse invariant of the bundle and that the image of the holonomy up to translation and bounded Hausdorff distance is a complete quasi-isometry invariant.

Proper holonomy happens frequently.    If the total space of the bundle is hyperbolic then by Gersten's converse to the theorem of Bestvina and Feign (\cite{gersten2}) the holonomy must be proper.   It also holds for many group extensions:

\begin{claim} If $\G$ is a finitely generated and tame group then $Out(\G) \to \QI(\G)$ is proper.
\end{claim}

\begin{proof}

It is clear there are only finitely many automorphisms with given quasi-isometry constants, simply look where the generators can go.    However the quasi-isometry constants for an automorphism may not be even approximately optimal for its equivalence class.   Indeed, all the inner autormorphisms are at bounded distance from isometries (left transations).   We need to see that this is essentially the only way this happens.

Fix $\phi \in Aut(\G)$ and suppose $f: \G \to \G$ is a $(K,C)$ quasi-isometry at bounded distance from $\phi$ with $(K,C)$ as small as possible.   Consider the collection of maps $ f_{\g}$ defined by $f_{\g}(x) = \phi(\g) f( \g \inv x)$.   These are clearly all $(K,C)$ quasi-isometries as well, and by assumption each is at bounded distance from $\phi$.   By tameness there is therefor an $R$ (depending only on $(K,C)$ and not on $\phi$) so that the $f_{\g}$ are within $R$ of one another.   Therefore, if we conjugate $\phi$ by $f(id)$ we get an automorphism whose constants are nearly optimal in its QI class.   As above, this implies there are only finitely many elements of $Out(\G)$ with $(K,C)$ representatives as claimed.
\end{proof}

Thus any extension in which the action of the quotient on the kernel (provided it is tame) is injective gives a bundle with proper holonomy, and therefore two such bundles are coarsely equivalent if and only if these images are at bounded Hausdorff distance after translation.

It is the "after translation" part of the conclusion that makes this difficult to understand - two subgroups of $Out(\G)$ (or of any discrete group) are at finite Hausdorff distance if and only if they are commensurable (in which case the extensions are commensurable as well).    Thus modulo the translation issue, such extension have a  strong type of rigidity.

The question therefor becomes, given a subgroup $G$ of $Out(\G)$, which quasi-isometries $f$ have the property that  for all $g \in G$, $f \circ g$ is at bounded distance from an element of $Out(\G)$?   Frequently one can show that for $G$ large enough, such and $f$ must be an automorphism.   For $\G$ a closed hyperbolic surface $S$, and $G = Out(\G) (= MCG(S))$ this is a result of Mosher, and is a key ingredient in our proof that $Aut(\G)$ is QI rigid (\cite{MW}).   Note that $Aut(\G)$ is a $\G$ bundle over $Out(\G)$, so the results here are directly relevant.

Another example comes from semi-direct products $\Z^n \by Q$ for $Q$ a subgroup of $\glnz$.    Here the fiber is not tame, but the assumed splitting allows us to reduce the structure group to $\glnr$ which is tame.   When $Q$ is large we can prove rigidty results using these arguments, and in particular theorem \ref{glnzbyzn}.    See section \ref{sec:pattern} for the details.

\section{Euclidean Fibers}\label{sec:notame}

We now turn to a more detailed study of bundles with fiber $\R^n$ (equivalently $\Z^n$) with various types of structure groups.   

For tame structure groups the results of the previous section give some useful classifications, but for non-tame structures things are more complicated.  For group extensions, this corresponds to he fact that extensions by groups with center are more complicated than simply the map to $Out$.  In that case, one needs to look at some $2$-dimensional cohomology classes with coefficients in the center of the fiber.   One can see how to start in that direction in our setting as well.  When $\Sigma$ is not tame, passing from $\Sigma$ to $\bar{\Sigma}$ throws away the information in $\Sigma_0$.   Given an $(F, \Sigma)$ bundle one can try to define a cocycle in ${H^1}_\infty (B, \Sigma)$ rather than ${H^1}_\infty (B, \bar{\Sigma})$.    Starting with closest point projections, one gets a uniform map $c:\Delta_r B \to \Sigma$.  This need not be a cocycle: the product around a loop is bounded distance from the identity, and hence in $\Sigma_0$, but need not vanish.   One would like to view this product as a uniform $2$-cocycle on $B$ with coefficients in $\Sigma_0$ and so get a class in ${H^2}_\infty(B,\Sigma_0)$ which measures the failure.   

Unfortunately this does not make sense.  Even if $\Sigma$ is a structure group so that using it as coefficients is meaningful, there are problems arising from the fact that $\Sigma_0$ need not be Abelian.  As cohomology with non-Abelian coefficients only makes sense for $H^0$ and $H^1$,  it is difficult to make sense of $H^2(B,\Sigma_0)$.  
 
We begin with a discussion of the translational case, $\Sigma=\R^n$ acting on $\R^n$ in the standard way.  Here $\Sigma_0 = \Sigma$, and $\bar{\Sigma}$ is trivial, so the previous invariants give no information.  The basic example of bundles with translational structure group come from central extensions: let 

$$ 1 \to A \to \G \to B \to 1 $$

be a short exact sequence of finitely generated groups, with $A$ in the center of $\G$.   For each $b \in B$ choose $\hat{b} \in \G$ with $\hat{b}A = b$.   This gives isometric identifications of the fibers with $A$, by left translation by the $\hat{b}$.   For any $b$ and $b'$ in $B$, let $\delta$ be a minimal norm element in $b^{-1}b' A$.  Right translation by $\delta$ gives a closest point projection map from the fiber over $b$ to the fiber over $b'$.  Under the given identifications of the fibers with $A$, this becomes a self-map of $A$, given by translation by $\hat{b'}^{-1}\hat{b}\delta$, which is an element of $A$.  Thus a central extension naturally gives rise to a bundle with translational structure group.

Let $X \to B$ be any bundle with fiber $A$ and translational structure group.  For any three points $b_1$, $b_2$, and $b_3$, the composition of projection from $b_1$ to $b_2$, $b_2$ to $b_3$ and $b_3$ to $b_1$ gives a translation of $A$, which we call  $\tau(b_1,b_2,b_3)$.   Since the closest point projections between fibers are well-defined to within a disctance controlled by the distance between the fibers, we know $\tau(b_1,b_2,b_3)$ has size bounded in terms of the diameter of $\{b_1,b_2,b_3\}$.    The following shows that the classification of coarse translational bundles closely parallels the classification of central extensions.

\begin{lemma} \label{translational} The map $\tau$ induces a injection $Ext(B,(A,A)) \to {H^2}_\infty(B,A)$.  The image of $\tau$ is the kernel of the map ${H^2}_\infty(B,A) \to {H^2}_c(B,A)$, in particular,  if $B$ is uniformly simply connected, $\tau$ is a bijection.
\end{lemma}

\begin{proof}

We first note that by definition for each $b$ and $b'$ the projection map from the fiber over $b$ to the fiber over $b'$ is a translation, $\phi_{b,b'}$.  Without loss of generality we can assume $\phi_{b',b} = - \phi{b,b'}$.  Thus $\phi$ is a coarse $1$-cochain on $B$ with values in $A$, although it need not be an $L^\infty$ coarse cochain.   We then have $\tau=d \phi$ by definition.  This shows that $\tau$ is a coarse $2$-cocycle on $B$ with values in $A$, and as observed above, it is also $L^\infty$.   Thus there is a cohomology class $[\tau] \in {H^2}_\infty(B,A)$.  As $\tau$ is, by construction $d \phi$, $[\tau] = 0$ in ${H^2}_c(B,A)$.

Surjectivity is immediate: let $\tau$ be any $L^\infty$, $2$-cocycle with $\tau = d \phi$ for some coarse $1$-cochain $\phi$.  Since all translations are isometries, and $d \phi$ is $L^\infty$, we can use $\phi$ to define a translational bundle $X_\phi$ over $B$.   By construction $\tau(X_\phi) = \tau$. 

If $X$ and $X'$ are two coarse $(A,A)$ bundles over $B$ such that, for some choices of respresentatives of their $(A,A)$ structures we have $[\tau(X)] = [\tau(X')]$ then there is a coarse  $L^\infty$ $1$-cochain $\beta$ with $d\beta = \tau - \tau' = d \phi - d \phi ' = d( \phi - \phi')$.  This means that $\beta + c= \phi - \phi' $ for some coarse $1$-cocycle $c$.  As $B$ is coarsely connected, $c = df$ for some function $f: B \to A$.  This $f$ determines a bundle map $F:X \to X'$ which is translation by $f(b)$ from $X_b \to X'_b$.    Since $\phi-\phi' = df + \beta$ and $\beta$ is $L^\infty$, the map $F$ commutes with closest point projections to within bounded distance, and so is an isomorphism of coarse $(A,A)$ bundles.  So the map $\tau:Ext(B,(A,A)) \to {H^2}_\infty(B,A)$ is injective.

Conversely, if $F: X \to X'$ is an isomorphism of $(A,A)$ bundles, we can view $F$ as a map $f: B \to A$ recording the fiberwise translations.  The fact that $F$ commutes with closest point projections to within bounded distance then implies $(\phi - \phi') - df = \beta$ is an $L^\infty$ $1$-cochain.   Then $d \beta = d(\phi - \phi') = \tau - \tau'$, so $[\tau] = [\tau']$ in ${H^2}_\infty(B,A)$.  So the map $\tau:Ext(B,(A,A)) \to {H^2}_\infty(B,A)$ is well-defined.
   
\end{proof}

Thus the classification of coarse translational bundles closely parallels that of central extensions, with periodicity replaced by boundedness.   However, we are frequently interested in coarse equivalence of bundles, not only equivalences that respect the translational structures.   Similarly for more general bundles with other structure groups.   Surprisingly the two are often equivalent:

\subsection{Structural Rigidity}

\begin{defn} Let $\Sigma$ be a structure set on $F$.  We say that $\Sigma$ is {\bf rigid} if any two $(F,\Sigma)$ bundles which are equivalent as coarse fibrations are equiavlent as $(F,\Sigma)$ bundles.  If $\Sigma'$ is a structure set with $\Sigma \subset \Sigma'$ we say that the pair $(\Sigma,\Sigma')$ is {\bf rigid} if any two $(F,\Sigma)$ bundles which are equivalent as coarse fibrations are equivalent as $(F,\Sigma')$ bundles. 
\end{defn}

\begin{lemma}\label{diff1} The structure group $(\R^n, \glnr)$ is rigid.
\end{lemma}

\begin{proof}

Let $X \to B$ and $X' \to B$ be $(\R^n,\glnr)$ bundles. Assume we have a quasi-isometry $f:X \to X'$ over $B$.  We may assume that $f$ takes the origin in each fiber of $X$ to the origin in the corresponding fiber of $X'$.   Define the maps $f_t(e)=\frac{1}{t} f(te)$, where the scalar multiplication takes place within fibers, and hence is well defined.    Along the fibers, if $f$ is a $(K,C)$ quasi-isometry then $f_t$ is a $(K,\frac{C}{t})$ quasi-isometry, and they all preserve
origins.  Further, since the closest point projections are linear, they commute with scalar multiplication, and hence the map $f_t$ are uniform bundle equivalence $X \to X'$.  There is a sequence $\{t_i\}$ going to infinity for which the maps $f_t$ converge to bundle equivalence  $F:X \to X'$ which is a bilipschitz map on each fiber.

As bilipschitz maps of $\R^n$ are differentiable almost everywhere, we can find a lipschitz section
$s$ of $X$ so that $F$ is fiberwise differetiable at $s(b)$ for all $b \in B$.  Since the gluing maps are linear, and linear maps commute with differentiation, these deriviatives assemble to an equivalence of $X$ and $X'$ as $(\R^n,\glnr)$ bundles.

\end{proof}

Most groups do not have $\glnr$ structures naturally, as it implies a canonical choice of basepoint in each fiber.   However, this is the only obstruction in reducing an affine structure to a linear one.

\begin{lemma}\label{lemma:section}
The structure group pair $((\R^n,\glnr), (\R^n, Aff(\R^n)))$ is rigid.
\end{lemma}

\begin{proof}

If $X$ has a $\glnr$ structure then the "zero section" gives a lipschitz section of the bundle.    Conversely, any  lipschitz section gives a reduction of the affine structure group to $\glnr$.    Any two Lipschitz sections differ by a lipschitz family of fiberwise translations and hence define equivalent $\glnr$ structures.  Thus if an affine bundle has a $\glnr$ reduction, it is canonical (this is really just the observation that the linear part of an affine map is well defined), and such a reduction exists if and only if there is a Lipschitz section.   Since the existence of such a section is a coarse invariant of the bundle the result follows.

\end{proof}

For translational structures things are not quite a simple, however one has:

\begin{lemma}\label{central} The structure group pair $((\R^n,\R^n), (\R^n, Aff(\R^n)))$ is rigid.
\end{lemma}

\begin{remark}
This should be compared to the result of Gersten \cite{gersten} which says that for any $G$, a central extension defined by a bounded cocyle is quasi-isometric (preserving fibers) to the trivial
extension.  Our result is both a generalization and a converse.   To see a the relationship, one must observe that for any group the map from the bounded cohomology of $G$ to the $L^\infty$ cohomology is trivial.  It may be true that the kernel of the map $H^2(G,Z) \to {H^2}_\infty(EG,Z)$ is precisely the image of bounded cohomology.  There are two cases where this is all understood - if 
$G$ is amenable then averaging shows both that the bounded cohomology vanishes, and that the 
map $H^2(G,Z) \to {H^2}_\infty(EG,Z)$ is injective.  At the other extreme, when $G$ is Gromov hyperbolic, the map from bounded cohomology to group cohomology is surjective (\cite{Gromov:boundedcohomology}) and the $L^\infty$-cohomology vanishes (\cite{gersten2}).
\end{remark}

\begin{proof}

Suppose $E$ and $E'$ are translational bundles over $B$ and that $f:E \to E'$ is a fiber preserving quasi-isometry.   We need to show $E$ and $E'$ are affinely equivalent.   Choose origins arbitrarily for the fibers of $E$, and choose their images under $f$ as origins for the fibers of $E'$.

Let $l$ be a loop in $B$ based at $b$.   Summing the gluing maps around $l$ produces a self-translation, $\tau_l$, of the fiber over $b$.  Likewise there is a translation $\tau'_l$ defined by the gluing maps of $E'$.  Since $f$ is a fiberwise quasi-isometry, it commutes with closest point projections to within bounded distance, and hence these diverge linearly as we go around $l$.  in other words, there is a $C$ (depending only on QI constants of $f$), so that within the fibers over $b$, one has $d(f_b(v+\tau_l),f_b(v)+\tau'_l) \leq C |l|$.  

\begin{lemma} Let $\{\tau_i, \tau'_i, r_i\}$ be a sequence of pairs of vectors in $\R^n$ and real numbers such that  there is a self-quasi-isometry $f$ of $\R^n$ such that for all $v$ and $i$, $d(f(v+\tau_i),f(v)+\tau'_l) \leq  r_i$.  There is a linear map, $T$, of $\R^n$ such that $d(T(v+\tau_i),Tv+\tau'_l) \leq  r_i$, with the distortion of $T$ bounded in terms of the quasi-isometry constants of $f$.
\end{lemma}

\begin{proof}

Notice that as $d(f(v+\tau_i),f(v)+\tau'_l) \leq  r_i$, $d(f(v+n\tau_i),f(v + (n-1)\tau_i)+\tau'_l) \leq  r_i$ and so $d(f(v+n\tau_i),f(v)+n\tau'_l) \leq  nr_i$.   Now form the rescaling procedure as before: $f_t(v)=\frac{f(tv)}{t}$.  We have $d(f_t(v+c\tau_i),f(v)+c\tau'_l) \leq  cr_i$ for all $c \in \frac{1}{t}\Z$.  Passing to a sequence of $t_i \to \infty$ for which the $f_t$ converge to a bilipschitz homeo $F$ of $\R^n$, we have $d(F(v+c\tau_i),F(v)+c\tau'_l) \leq  cr_i$ for all $c \in \R$.  Since $F$ is bilipschitz, it is differentiable almost everywhere, and hence somewhere.  Applying $d(F(v+c\tau_i),F(v)+c\tau'_l) \leq  cr_i$ for all $c \in \R$ with small $c$ and $v$ a point of differentiability of $F$ shows $|dF_v(\tau_i) - \tau'_i| \leq r_i$.  Thus $dF_v$ is the desired linear map $T$. 

\end{proof}

\begin{remark}
Note that if $\tau_i$ is large compared to $r_i$, this condition says almost exactly what $T$ must do to $\tau_i$.  If there are many $\tau_i$ like this, there will only be a single linear map satisfying the conclusion of the lemma.  The fact that $F$ is almost everywhere differentiable with that map
as its derivative implies $F$ (resp. $f$) is affine (to within bounded distance) with that linear part.
This implies that every fiberwise quasi-isometry is affine.  This is our first glimpse of the phenomenon of pattern rigidity, which we will discuss extensively later.
\end{remark}
\end{proof}

Thus, we can change $E'$ by a linear map and arrange that there is a $C$ so that $|\tau_l - \tau'_l| \leq C |l|$ for all loops $l$.  In terms of the classifiying cohomology classes, this says that $c_E - c_{E'} = db$ where $|\Sigma_{e \in l} b(e)| \leq C |l|$.  This condition on $b$ would clearly hold if $b$ were $L^\infty$.  The converse is almost true:

\begin{theorem}\label{vanish} Let $c$ be an $L^\infty$ $2$-cocycle on $B$.  The class of $c$ vanishes in ${H^2}_\infty(B)$ if and only if there is a $C$ such that for every loop $l$ in $B$ (based at some point $b$), $|c(l)| \leq C|l|$.
\end{theorem} 

\begin{remark}

Since $c$ is $L^\infty$, $|c(l)|$ is bounded by a constant times the area of a disk with boundary $l$.   The theorem says that $c$ represents the trivial class if and only if $c(l)$ is bounded by a multiple of $l$.  If (and only if) $B$ is Gromov hyperbolic, $l$ must bound a disk of area bounded by a multiple of $|l|$, and hence the $L^\infty$ cohomology vanishes.  That special case of the theorem is a theorem of Gersten \cite{gersten2}, and our proof of the more general result is essentially unchanged from his proof in the hyperbolic case. 

\end{remark}

\begin{proof}

We can reduce to the case of $n=1$ simply by working one coordinate at a
time.  Th only if part of the theorem is trivial, so assume we have a $2$-cocyle
$c=da$ such that for every loop $|\Sigma_l a(e)| \leq C |l|$.  Fix a base point $b_0$ in $B$.

For any $b \in B$, define $f(b) = sup \{ a(p) - 2C |p| : p \hbox{a path from $b_0$ to $b$}\}$.   If we let $p$ be any path $b_0$ to $b$, then for any other path $p'$ with the same endpoints, $p \cup p'$ is a loop at $b_0$, and the assumption that
$|a(p \cup p')| \leq C (|p| + |p'|)$ shows that $a(p') -2C |p'| \leq a(p) + C(|p|-|p'|)$, and hence that the supremum is finite, so $f$ is well defined.  

Let $e=(b,b')$ be an edge of $B$ and let $p$ and $p'$ be paths from $b_0$ to $b$ and $b'$ which achieve the suprema defining $f(b)$ and $f(b')$ within $\varepsilon >0$.   Let $l$ be the loop $p \cup p' \cup e$.   Let $p_1$ be the path to $b$ which is the concatenation of $p'$ and $e$, and likewise let ${p'}_1$ be
$p \cup e$.   

We have $|f(b')-f(b)-a(e)| \leq K$. By definition $a(p_1) - 2C |p_1| \leq a(p) - 2C|p| + \varepsilon$ and $a({p'}_1) - 2C |{p'}_1| \leq a(p') - 2C|p'| + \varepsilon$.     Adding these equations, $a(p_1) + a({p'}_1) - 2C (|p_1| + |{p'}_1|) \leq a(p) + a(p') - 2C(|p| + |p'|) + 2\varepsilon$.  Expanding the left hand side gives $a(p)+a(p') - 2C(|p|+|p'|+2)$

Thus $a+df$ is an $L^\infty$ cocycle, and $d(a+df)=da=c$, so $c$ is trivial in ${H^2}_\infty(B)$.

\end{proof}

\begin{corollary}  Coarse translational $\R^n$ bundles $X$ and $X'$ over $B$ are equivalent as coarse fibrations if and only if $[\tau(X)]$ and $[\tau(X')]$ are in the same $\glnr$ orbit inside ${H^2}_\infty(B,\R^n)$.  In particular $X$ is equivalent to the product $B \times \R^n$  if and only if $\tau(X) = 0$.
\end{corollary}

For central extensions this proves theorem \ref{thm:central}.

\begin{remark} The identification of those central extensions coarsely equivalent o products was obtained previously in \cite{KleinerCentral}.  This is somewhat easier than the general case as one does not need theorem \ref{vanish}, and instead can argue directly that a translational bundle is trivial if and only if it has a lipschitz section, which is equivalent to the vanishing of $\tau$.
\end{remark}

This gives a reasonably effective classification of those coarse $\R^n$ bundles which can be reduced to have $\glnr$ or translational structure groups.   One would like to cover all affine structure groups as this is what is relevant for simplicial actions.

\begin{question} Is $(\R^n,Aff(\R^n))$ rigid?
\end{question}

This is much like the situation for group extensions of $B$ by $A$ with $A$ Abelian; the coefficients correspond to the map $B \to Out(A)$ and extensions of that type are classified by $H^2(B, A)$ with $A$ viewed as the appropriate $B$-module. We can make some progress on general affine bundles along a similar outline using the previous ideas.  Associated to every affine bundle is a "local coefficient bundle", which is an $(\R^n, \glnr)$ bundle.  The same argument as in  lemma \ref{diff1} shows that if two affine bundles are equivalent as coarse fibrations, then their coefficient bundles are $\glnr$ equivalent.   Fixing a $\glnr$ bundle as the coefficient bundle, we are left with classifying affine bundles with those coefficients.   As affine bundles, this classification is given by a twisted coefficient version of lemma \ref{translational}.   The missing ingredient is the analogue of theorem \ref{vanish} for  twisted coefficients (with boundedness formulated appropriately).   For bundles whose linear part is sufficiently complicated this can be shown using the techniques of the next section, but in that setting stronger results are available.

\subsubsection{Strong Structural Rigidity:Pattern rigidity and Affine Structure Group}\label{sec:pattern}

There is another approach to these rigidity questions which gives stronger results : every equivalence between $X$ and $X'$ as coarse fibrations is an equivalence as $(F, \Sigma)$ bundles.   This holds when the "dynamics" of the holonomy are sufficiently complicated.   In particular, we will see that for affine bundles with sufficiently complicated linear parts, this will hold.  The idea is to define certain foliations of the fibers that are intrinsic to the dynamics and so must be preserved by any quasi-isometries.  If there are enough of these, one can show the quasi-isometries are in $\Sigma$.    This is essentially the method used by Farb and Mosher in \cite{FMabc} and \cite{FMsbf} to prove quasi-isometric rigidity for certain Abelian-by-Cyclic and Surface-by-Free groups.   We will focus on the affine case, although the case of mapping class groups on $\hyp^2$ lurks behind our result with Mosher for QI rigidity for automorphsims of surface groups (\cite{MW}). 

Let $X \to B$ be an $(\R^n, Aff(\R^n))$ bundle.  For any path $p=b_0, b_1, \cdots $ in $B$, a {\bf $K$-straight lift} of $p$ is a $K$-path $x_0, x_1, \cdots $ in $X$ with $x_i$ in the fiber over $b_i$.   Two paths in $X$ covering $p$ are called {\bf parallel} if for some $R$, $d(x_i, x'_i)\leq R$ for all $i$.    If the projection from $X_{b_i}$ to $X_{b_{i+1}}$ is $\phi_i$ then straightness of a lift corresponds to the "drift" of $x_{i+1}$ from $\pi_i (x_i)$.  

For $u,v \in X_{b_0}$ define $D_p (u,v)$ as the minimal $K$ such that there are $K$-straight lifts of $p$ beginning at $u$ and $v$ which are parallel.  

\begin{lemma} There is a semi-norm $||-||_p$ on $\R^n = X_{b_0}$ so that $D_p(u,v)$ is coarsely equivalent to $|| u -v ||_p$. 
\end{lemma}

\begin{proof}

Consider the sequence of vector spaces $X_{b_0}, X_{b_1}, \cdots$.  By changing the identifications of each with $\R^n$ be a translation (in particular, by an isometry), we can arrange for the closest point projection from $V_i = X_{b_i}$ to $V_{i+1} = X_{b_{i+1}}$ to be a linear map, $T_i$.   Then $D_p(u,v)$ for $u,v$ in $V_0$ is clearly coarsely equivalent to:

$$\hbox{inf} \{ C : \exists \{u_i\},\{ v_i\} \in V_i, r \in \R : \forall i d(u_i,v_i) \leq r , d(T_i u_i, u_{i+1}) \leq C, d(T_i v_i, v_{i+1}) \leq C\}$$

which, in turn is coarsely equivalent to $N(u-v)$, where:

$$N(u) = inf \{ C : \exists \{u_i\}  \in V_i, r \in \R : \forall i ||u_i|| \leq r , d(T_i u_i, u_{i+1}) \leq C \}$$

Trivially we have that $N(u) \geq 0$ for all $u$.   Considering the sequence with $u_i =0$ for all $i>0$ we see that $N(u) \leq || T_0 u || < \infty $.  Since the $T_i$ are linear, it is immediate that $N(tu) = t N(u)$ and $N(u+v) \leq N(u) + N(v)$, and so $N$ gives a semi-norm as claimed. 

\end{proof}

For any path $p$ let $\lambda_p$ be the foliation of $X_{b_0}$ by the affine subspaces with $D_p = 0$.   If $X$ and $X'$ are two affine bundles over $B$, any equivalence as coarse fibrations must coarsely respect the functions $D_p$ for all $p$.   In particular, the equivalence must carry the leaves of the foliation $\lambda_p$ to within bounded distance of the leaves of $\lambda'_p$.    For sufficiently large collections of such foliations this forces $f$ to be at bounded distance from an affine map.    It is not clear precisely what "sufficiently large" means here, but certainly containing $d+1$ hyperplanes or $d+1$ lines in general position in $\R^d$ is enough (\cite{MSW2}). 

\begin{corollary}\label{patternrigid}  Let $X$ be an $(\R^n, Aff(\R^n))$ bundle over $B$.   If the collection $\Lambda(X)$ of affine foliations $\lambda_p$ for paths $p$ starting at $b_0$ is sufficiently large, then any equivalence of $X$ to another affine bundle over $B$ as coarse fibrations is an equivalence of affine bundles.
\end{corollary}

For this to be useful, we need to have examples where $\Lambda(X)$ is sufficiently large.   If $X$ arises as from an action of a group $\G$ on $B$ (as in lemma \ref{lemma:action}) then this reduces to the image of $\G$ in $\glnr$ being dynamically complicated.    One example which is certainly big enough is $\glnz$ acting on $\R^n$ as the powers of any semisimple element give a line (the eigenspace for the larges eigenvalue) and the translates under $\glnz$ give the required collection of $d+1$ lines in general position.   We are almost ready to prove Theorem \ref{glnzbyzn}.  We need:

\begin{lemma}\label{ratnerlite} If $T \in \glnr$ is such that $T \glnz T\inv$ is at finite Hausdorff  distance from $\glnz$ then $T \in PGL_n(\Q)$.
\end{lemma}

\begin{proof}

Consider the set of vectors $v$ for the set $v A$ for $A$ in $\glnz$ does not accumulate at zero.   We claim that $v$ is precisely the set of vectors which are on rational lines.   To see this, consider any vector $v=(x_1, \ldots, x_n)$.   If $|x_i|$ is the largest coordinate, and $x_j \neq 0$ then elementary matrices will find another element in the same $\glnz$ orbit with only the $i$-coordinate changed, and with new $i$-th coordinate smaller than $|x_j|$.    In this way we see that the norm can be reduced by fixed factor unless at most one coordinate is non-zero.    The $\glnz$ orbit of such vectors is as claimed.

Since $T \glnz T\inv$ is at bounded Hausdorff distance from $\glnz$, its orbits track those of $\glnz$ - in particular it must preserve the set of rational lines by the characterization in the previous paragraph.   Thus $T$ in $PGL_n(\Q)$.

\end{proof}

\begin{theorem} Let $\G = \Z^n \rtimes GL_n(\Z)$.  The quasi-isometry group of $\G$ is ${H^1}_\infty(GL_n(\Z), \R^n) \rtimes PGl_n(\Q)$.  Any groups quasi-isometric to $\G$ is commensurable to an extension of $\Z^n$ by a finite index subgroup of $GL_n(\Z)$ which acts in the standard way on $\Z^n$
\end{theorem}

\begin{remark} This almost says that any group quasi-isometric to $\G$ is commensurable to it, but  there is an extension issue we do not know how to resolve - does $H^2(\Lambda, \Z^n)$ vanish for all $\Lambda$ finite index in $GL_n(\Z)$?  Even if it does not one still may have no examples of  such extensions which are quasi-isometric (that would involve vanishing in some appropriate $L^\infty$ sense. 
\end{remark}

\begin{proof}

The group $GL_n(\Q)$ acts on $\G$ as its group of commensurators, and ${H^1}_\infty(SL_n(\Z), \R^n) $ is the group of quasi-isometries of $\G$ which cover the identity on $Sl_n(\Z)$ and are translations along the fibers up to bounded distance.

We first claim that any quasi-isometry of $\G$ is an automorphism of the coarse fibration structure.   The results of \cite{MW} show that this holds for any fibration with fibers satisfying coarse Poincare duality and where the base has enough top dimensional holoology classes to separate point.   In this case the maximal nilponent subgroups and their cosets provide such branching top dimensional submanifolds. 

By theorem \ref{patternrigid} any such map is bounded distance from an affine map on fibers.  This affine map must move the image of the holonomy a bounded distance from itself.  By lemma \ref{ratnerlite} the linear part is therefor in $PGL_n(\Q)$.   This gives the map to $PGL_n(\Q)$.  As the holonomy is proper, the kernel is as claimed.

If $H$ is quasi-isometric to $\G$ then $H$ quasi-acts on $\G$, and hence there is
a homomorphism $H$ to ${H^1}_\infty(SL_n(\Z), \R^n) \rtimes Gl_n(\Q)$.   Since the quasi-isometry constants are uniform, the image in $Gl_n(\Q)$ is commensurable with $SL_n(\Z)$.   Further, the kernel of this map to $SL_n(\Z)$, which lives in ${H^1}_\infty(SL_n(\Z), \R^n)$ quasi-acts properly and cocompactly on the fibers, and on each fiber is a group of translations.  Thus the kernel of the map is virtually $\Z^n$ acting in the standard way.  
\end{proof}

Note here that we do not use the quasi-isometric rigidity of $\glnz$ at all here, indeed this proof works for $n=2$ where rigidity certainly fails. 

\section{Bundles over Trees}\label{sectreebase}

Note:  to keep verbiage under control we use {\bf tree} in this section to mean a bushy tree of bounded valence.

When the base of our bundle is a tree, we can get more precise information.  The main reason for this is that, being one dimensional, none of the phenomena involving ${H^2}_\infty$ arise.    If $\G$ acts coboundedly on a tree $T$ of finite valence, then  $\G$ has the structure of a coarse bundle over $T$, with fibers the vertex stabilizer (recall that any two are commensurable).  This is equivalent, via Bass-Serre theory (\cite{BassSerre}), to saying that $\G$ has a decomposition as a graph of groups in which all the edge group inclusions have finite index image.  Since this means all the edge and vertex groups are in a given commensurability class, we call such graphs of groups {\bf homogeneous}.  Some examples are:

\begin{itemize}

\item Semi-direct products with free groups: For any group $H$, and any map $\alpha:F \to Aut(H)$, where $F$ is a free group, the semi-direct product $H \rtimes_\alpha F$ is a homogeneous graph of groups with all vertex and edge stabilizers equal to $H$.

\item The Baumslag-Solitar groups $BS(m,n)=<a,b|a^{-1}b^ma=b^n>$ is a homogeneous graph of $\Z$s.  More generally, the mapping torus of any commensurator is an homogeneous graph of groups. 

\end{itemize}

While we state our results for any coarse locally compact structures, the main results are for homogeneous graphs of $\Z^n$s, and the reader who assumes we are in that case throughout will not miss much of importance.   There may be other examples, but groups of commensurators are not well understood.

\begin{question}  For which groups $\G$ is $Comm(\G)$ coarse locally compact?  
\end{question}

Of course, if $Comm(\G)$ is small, for example just $Aut(\G)$, then it is.  In this case the map to $\QI$ is proper, and so the situation is simpler (essentially $Comm(\G)$ is then a discrete subgroup of $\QI$).   In that setting much can be done simply by the methods of section \ref{section:proper}.  We discuss one example of this, non-uniform lattices, below.    The situation for $\G$ the fundamental group of a closed hyperbolic surface would be interesting to understand, as there all the other hypotheses needed for an analogue of Theorem \ref{thm:graphofzn} hold, but $Comm(\G)$ is quite large and mysterious.
 
In studying homogenoues graphs of groups, we can clearly apply many of the results of this paper.   One of the things which makes this cases easier is that the structures are cocompact: there is a group $\G$ which acts cocompactly preserving all the bundle data.  In particular, for bundles with structure group $G$, there is a homomorphism $\rho: \G \to G$ and the holonomy map is then equivariant, $h(\g x) = \rho(\g) h(x)$.   This, together with the relative simplicity of group actions on trees, makes the situation intelligible.  
\
\begin{defn}  A {\bf halfspace} of $T$ is a component of the complement of an edge.  We say that a halfspace $H$ of $T$ {\bf carries the holonomy} if the image of $T$ in $\hat{G}$ is contained in a neighborhood of the image of $H$.  
\end{defn}

For any edge $e$, one of three things happens: both its complementary halfspaces carry the holonomy, only one does, or neither does.   

\begin{lemma} \label{thm:trichotomy}Let $\G$ be a homogeneous graph of groups with Bass-Serre tree $T$.  Then precisely one the followings holds:

\begin{itemize}

\item No halfspace carries the holonomy.
\item Every halfspace carries the holonomy.
\item For every edge $e$, exactly one of its complementary halfspaces carries the holonomy.

\end{itemize}
\end{lemma}

\begin{proof}

First consider the case where there is some edge $e_0$ both of whose halfspaces carry the holonomy.    For any halfspace $H$ of any edge we can translate $e_0$  to some $e_1$ in $H$, then since one halfspace of $e_1$ is a subset of $H$ and carries the holonomy, so $H$ carries the holonomy.   As $H$ is arbitrary we shown all halfspaces carry the holonomy, which is case 2.

Now assume no edge has both sides carrying the holonomy.   If no edge has any halfspace carrying the holonomy then we have the first case.    That leaves the case that some edge has a halfspace $H$ that carries the holonomy.  By cocompactness, every edge has a translate not in $H$.  As one of its complementary components then contains $H$, that component carries the holonomy.  Thus one halfspace of every edge carries the holonomy as required.

\end{proof}

\subsection{Parabolic Holonomy}

If exactly one halfspace for every edge carries the holonomy, we call the holonomy map {\bf parabolic}. 

\begin{proposition} Let $\G$ be a homogeneous graph of groups with Bass-Serre tree $T$ with parabolic holonomy.   There is a unique end $a$ of $T$ such that a halfspace carries the holonomy if an only if it contains $a$.  The endpoint is fixed by $\G$, and as a consequence, $\G$ is the mapping torus of a finite index image endomorphism (a finite ascending HNN extension).
\end{proposition}

\begin{proof}

We have an orientation of the edges which is $\G$ invariant; point the edges 
towards the halfspace which carries the holonomy.   

For every vertex of $v$ there is at most one edge at $v$ oriented away from $v$.  If $e_1$ and $e_2$ are edges at $v$, with $e_1$ oriented away from $v$, then the halfspace of $e_2$ which contains $v$ contains the halfspace of $e_1$ which carries the holonomy, and hence carries the holonomy.  Thus $e_2$ is oriented towards $v$.

This defines a "flow" on $T$.  Such a flow has a unique sink (if there were two, some vertex along the path between them would have both edges oriented out) which is either a vertex of $T$ or an end of $T$.   Since it is unique, the sink is $\G$ invariant and as $\G$ acts cocompactly there is no $\G$ invariant vertex, hence the sink is an invariant end as claimed.

Thus we get a $\G$ invariant orientation  on $T$, in which, at every vertex $v$ there is precisely one edge  oriented away from $v$.  This orientation descends to an orientation of 
the quotient graph with the same property.  Further, given an edge $e$
of the quotient, and a lift $\bar{v}$ of its initial point $v$, there can
only be one lift of $e$ at $\bar{v}$ which implies that the edge group of $e$
is equal to the vertex group at $v$.  Unless both endpoints of $e$ are 
the same, this allows us to collapse $e$ and obtain a smaller, equivalent, graph of 
groups representation of $\G$.  If we continue collapsing edges of
this sort until no more collapses are possible, we reach a graph of 
groups with a single vertex and a single edge, in which one of the edge
to vertex inclusions is an isomorphism.  This is precisely an ascending HNN extension.  

\end{proof}
  
\begin{remark}

A priori, not all ascending HNN extensions give rise to  bundles with parabolic holonomy.  Indeed it is neither clear  that one halfspace of each edge carries the holonomy nor that the other does not.  It is not clear that one cannot have parabolic holonomy when both inclusions are isomorphisms, the case of a mapping torus of an automorphism.

\end{remark}

The geometry of ascending HNN extensions is delicate.  Among the Baumslag-Solitar groups, this is precisely the subclass of solvable groups, and they are shown (\cite{FMbs1}) to be quite rigid, unlike the other Baumslag-Solitar groups,  which are sufficiently flexible that they are all quasi-isometric (\cite{WhyBS}).   Farb and Mosher have generalized their techniques to general ascending HNN extensions of $\Z^n$ (\cite{FMabc}).  The techinques of \cite{WhyBS} also extend to graphs of $\Z^n$ without parabolic holonomy, see section \ref{section:graphofzn}.

\subsection{Coarse locally compact structure groups}

The other two behaviors of holonomy map are only readily analyzable when the structure group is coarsely locally compact (see definition \ref{clc}).  

\begin{proposition} Let $\G$ be a homogeneous graph of groups with Bass-Serre tree $T$ such that no halfspace carries the holonomy.   If the corresponding coarse bundle over $T$ has coarsely locally compact structure group then the holonomy map, $h$, is proper.
\end{proposition}

\begin{proof}

Fix $(K_0,C_0)$ as in the definition of coarse local compactness for $\hat{G}$.

If $h$ is not proper, then there is a sequence $\g_i$ so that for all $v \in T$, $\g_i v$ leaves all compact sets and $p(\g_i)$ is bounded in $\hat{G}$.   By passing to a subsequence we may assume that for all $v$, $g_i v$ converges to a point $a$ in $\partial T$.     We claim that any halfspace containing $a$ carries the holonomy.    
  
Let $e$ be an edge, oriented to that $ a \in H_+$.    For any $v$ in $H_-$, consider the sequence $g_i v$.   We have $h(g_i v) = p(g_i) h(v) = h(v) p(g_i)^{h(v)}$.   The sequence $p(g_i)^{h(v)}$ is bounded in $\hat{G}$ and by coarse local compactness that means there is some element $g$ so that infinitely many of the $p(g_i)^{h(v)}$ are in $g \hat{G}_{(K_0,C_0)}$.    Pass to this subsequence.    

We then have $(p(g_i)^{h(v)})\inv p(g_j)^{h(v)} in \hat{G}_{(K_0,C_0)}$ for all $i$ and $j$.    This means $g_i \inv g_j v$ has holonomy within $ \hat{G}_{(K_0,C_0)}$ of $h(v)$.    For all $j >> i$ this point is in $H_+$, so we have $H_+$ carries the holonomy as claimed.
  
\end{proof}

 As we already have seen, when the holonomy is proper the classification problem for bundles reduces to understanding of the image of the the holonomy map.   The last case is more complicated.

\subsection{Folded Holonomy} 

Graphs of groups where all halfpsaces carry the holonomy, the {\bf folded holonomy} case, are at the opposite extreme from proper.  It is remarkable that their quasi-isometric classification works out the same way: in the proper case the lifting problem is trivial and has canonical solution, while the folded holonomy case is flexible enough to allow the lifting problem to be solved, but highly non-canonically.  We start by giving a more precise statement of how non-proper the holonomy map is in the folded case when the structure group is coarse locally compact.

\begin{defn}
 A map $f:T \to \hat{G}$ has {\bf directed path lifting} if there is $V$ uniform in $\hat{G}$ so that for all $U \subset \hat{G}$ uniform there is an $r$ so that for any edge $e$ of $T$ and $g \ in f(T)$ within  $U$ of $f(\iota e)$ there is a $v$ in ${T^+}_e$ within $r$ of $e$ with $f(v)$ within $V$ of $g$.
 \end{defn}
  
This differs from the definition of folded holonomy in that not only are the images of both halfspaces coarsely the same as the image of the whole tree but additionally the distance one must go into a halfspace to find a given image is controlled in terms of how far the desired image point is from the image of the edge.  The name {\bf directed path lifting} comes by comparison with the definition of a coarse fibration and the way the property is used below.  We first show it follows from flded holonomy assuming coarsely locally compact structure group.

\begin{lemma} If $T$ is the Bass-Serre tree of a bushy, homogeneous
graph of groups $\G$ for which the holonomy map $\G \to \hat{G}$ is folded, with $\hat{G}$ coarsely
locally compact, then the holonomy map has directed path lifting.
\end{lemma}
\begin{proof} 

Without loss of generality we may assume (notation as in the definition of directed path lifting) that $e$ is in a fixed (finite) fundamental domain of $T$.     If for each such $e$ the definition holds then the claim follows simply by taking the largest $r$ among all the edges (and the union of all  the $V$).   Thus we just need to see that for any edge $e$ there is a $V$ so that for all $U$ there is  an $r$ so that for any $g$ within $U$ there is a $v$ in $T^+$ within $r$ of $e$ with holonomy within $V$ of  $g$.    

By coarse local compactness such a $U$ is covered by a finite collection of translates of a fixed $U_0$.    We let $V$ be at least as big as this $U_0$ so we need to simply find a finite collection of points in the half space.   Since the holonomy is folded we know we can do this within some uniform set.   Take $V$ to be the union of this and $U_0$, and let $r$ be the suprema of the distances of these finitely many edges to $e$.

\end{proof}

We now aim to show that the coarse image of the holonomy (up to conjugation) is a complete invariant for bundles with folded holonomy.   The main theorem here is:
  
\begin{theorem}\label{lifting} Let $f_{i}:T_{i} \to \hat{G}$ be lipschitz 
maps.  Suppose both $f_{i}$ have directed path lifting and have images at finite Hausdorff distance, then  $T_{1} \to T_{2}$ are quasi-isometric over $\hat{G}$.
\end{theorem}

\begin{proof}

 The idea is to show that any $T \to \hat{G}$ with directed path lifting can be decomposed (over $\hat{G}$) to resemble a free product of any other tree with nearby image in $\hat{G}$.    To do this we need to "foliate" $T$ by subtrees quasi-isometric over $\hat{G}$ to the given model.     The following lemma is the tool for finding the necessary subtrees.

\begin{lemma}\label{lift}  Let $f:T \to \hat{G}$ be lipschitz and have directed path lifting. Given any $T'$ and a lipschitz map $p:T' \to G$ with $p(T)$ Hausdorff equivalent to $f(T)$ there is $U \subset \hat{G}$and constants $(A,B)$ so that  given any $v \in T$, any edge $e$ at $v$, and any $u \in T'$ with $p(u)$ in the $U$-neighborhood of $f(v)$ there is a subtree $T''$ of $T$,  with $v \in T''$, $e \notin T''$, and an $(A,B)$ quasi-isometry $T' \to T''$ commuting over $\hat{G}$ to within $U$ and taking $u$ within the $B$ neighborhood of $v$. \\

Further, the tree $T'$ can be chosen so that the edges adjacent to, but not in, $T''$ are $B$-dense in $T''$.  
\end{lemma}

Using this lemma, we prove Theorem \ref{lifting} as follows.  We first claim that there is a vertex covering of $T_{2}$ by disjoint subtrees, which are lifts of $T_{1}$ as in lemma and such that the maps from each subtree to  $T_1$ assemble to a Lipschitz map $T_2 \to T_1$.  

We build this inductively, covering larger and larger subtrees of $T_2$.  We can start the induction by finding a single subtree from the lemma.    Suppose then that we have covered some ${T'}_2$ in $T_2$.  If this is not $T_{2}$ then there is an edge $e$ with one endpoint in, and one endpoint not in ${T'}_2$.  We then build a lift of $T_{1}$ through the other endpoint of $e$ on the opposite side of $e$ from  ${T'}_2$, making sure the endpoint of $e$ maps close to the image of the vertex where it connects to ${T'}_2$.    Thus we can cover all of $T_2$ as claimed.    Further,  the lemma guarantees that the edges of $T_2$ not in any of these subtrees hit each subtree in a $B$-dense subset.

This foliation describes $T_2$ as quasi-isometric to one built as follows:

Let $Q$ be a homogeneous tree of infinite valence, with the edges at each vertices of $Q$ labelled by vertices of $T_{1}$.  One can builds a tree $T$ with vertices  $VT_{1} \times VQ$ and with edges glued in over the edges of $Q$, attached to the vertices whose $T_{1}$ components correspond to the labelling. 

Two such trees are quasi-isometric over $\hat{G}$ if there is an isomorphism between $Q$ and $Q'$ which is at uniform bounded distance from the identity with respect to the labellings.  We can build such an isomorphism inductively simply by producing bijections between the edges at a vertices of $Q$ and one of $Q'$ one vertex at a time.  Appropriate bijections exist by \cite{thesis} as we know the labellings correspond to $B$-dense subsets of $T_1$.  Thus any two such trees are quasi-isometric over $\hat{G}$.   In particular, applying the same arguments to $T_{1}$ mapping to itself shows $T_{1}$ also has this form and this proves the theorem.

\end{proof}

We now prove Lemma \ref{lift}.  

\begin{proof}

Build the lift inductively on the ball of radius $n$ in $T_{1}$ with the further property that the map is injective on vertices and takes the sphere of radius $n$ to extreme
points of the image subtree.  To extend to radius $n+1$, take $r_{0}$  large enough so that any hemisphere (sphere around an edge intersected with one halfspace of that edge) of radius $r_{0}$ has strictly larger cardinality than the valence of $T_{1}$.  For each vertex $v$ in the sphere or radius $n$, arbitrarily map the adjacent vertices in the sphere of radius $n+1$ into the hemisphere at $f(v)$ avoiding the edge connecting to the image of the ball of radius $n$.  Now, starting with the edge leaving the sphere, use highly non-properness to find a short path from there to a point which maps near where it belongs in $\hat{G}$. As there where extra vertices in the hemi-sphere which we did not use, we know that the missed edges are dense, as needed.

\end{proof}

The results of this section may have more general applications in 
the geometric study of finitely generated groups than the ones given 
here.  We note the following two consequences:

\begin{theorem} Let $\G$ be a finitely generated group which is 
not free.  Let $S_{1}$ and $S_{2}$ be finite generating sets, then 
there is a quasi-isometry $F<S_{1}> \to F<S_{2}>$ commuting up to 
bounded distance with the maps to $\G$.
\end{theorem} 

\begin{theorem} Let $\G_i$ be finitely generated non-free groups, 
and $f:\G_{1} \to \G_{2}$ a quasi-isometry.  For any $S_{i}$ 
finite generating sets there is a quasi-isometry $F<S_{1}> \to 
F<S_{2}>$ which covers the given quasi-isometry.
\end{theorem} 

In both of these results one can relax non-free to free with 
non-minimal generating set.  

\section{Applications}\label{section:graphofzn}

The results of the previous sections let us prove Theorem \ref{thm:graphofzn} classifying homogeneous graphs of $\Z^n$s.   First we observe that by the results of \cite{MSW1} any group quasi-isometric to such a graph of groups is itself one, and that all quasi-isometries between such groups are automatically fiber preserving.    The classifying bundles classifies the groups.

All these bundles can have their structure groups reduced to $\glnr$ as they naturally have affine structure group and as bundles over trees always have coarse sections.    Further, by lemma{glnrstructurerigid}, the classification then reduces to the classification as coarse $\glnr$ bundles.    We have already seen that this implies the image of the holonomy map, up to Hausdorff equivalence, is a quasi-isometry invariant.   

For bundles within a given holonomy type, the trichotomy theorem above says that these divide into three classes: proper, parabolic, and folded.   In the first and last case there is at most one quasi-isometry type.   The parabolic case occurs precisely for ascending HNN extensions, and these are classified in \cite{FMabc}.   All that remains is:

\begin{lemma}  If $\G$ is the fundamental group of a graph of $\Z^n$s with proper holonomy then $\G$ is virtually a semi-direct product $\Z^n \by F$ for $F$ a free subgroup of $\glnz$.
\end{lemma}

\begin{proof}

This amounts to the claim that proper holonomy means that the stabilzer of any vertex of the Bass-Serre tree fixes the entire tree.     Suppose $g$ fixes a vertex $v$ but does not fix an adjacent edge $e$.   Let $H_1$ be the half of the tree on the opposite side of $e$ from $v$ and let $H_2$ be the  half of the tree on the opposite side of $ge$ from $v$.    By definition $gH_1 = H_2$.  As $g$ is a vertex stabilizer, hence a translation, the holonomy maps $g$ invariant, hence $H_1$ and $H_2$ have the same image under the holonomy map.    As both are infinite subtrees this contracts properness of the holonomy map.

\end{proof}

Another case where these results apply directly is to homogeneous graphs of non-uniform lattices (other than in $SL_2 \R$).     Let $\Lambda$ be such a lattice in a semisimple Lie group $G$.  By Mostow no such $\Lambda$ is conjugate to a proper finite index subgroup of itself, so there are no ascending HNN extensions with the fiber group, and correspondingly we need only consider the proper and folded holonomy cases.

As discussed previously (example \ref{examples:structures}) the quasi-isometry group of $\Lambda$ is virtually the commesurator inside $G$.  So from any homogeneous graph of $\Lambda$s we get a homorphism to $Comm(\Lambda)$.  Thus the holonomy gives a well defined commensurability class of  finitely generated subgroup $H$ of $Comm(\Lambda)$ which virtually contains $\Lambda$.   If there is infinite kernel then the holonomy map is clearly folded.    Thus in the proper case we know quasi-isometric means commensurable.   

In the folded case we still have a well defined commensurability class of group $H$, and one example of such a bundle is $H *_ {\Lambda} (\Lambda \times \Z)$.   In other words we simply take $H$ and add elements which commute with $\Lambda$ and have no other restrictions.    Thus all the groups in the folded holonomy case are quasi-isometric to such an amalgamated product.   

It is clear than any such $H$ can occur  - choose generators for $H$ over $\Lambda$ and build a  rose of groups with those generators on the edges.    These subgroups are actually quite limited however - for example, for $\Lambda=SL_n(\Z)$ it follows from a result of Venkataramana (\cite{Venkataramana}) that any such subgroup is commensurable to $Sl_n(Z[\frac{1}{S}])$ for some $S$.    This means, in particular, that the image of the holonomy map is just the product of the corresponding buildings and so the holonomy cannot be proper (as such a product is not quasi-isometric to a tree!).   Indeed the only case where proper holonomy seems possible is in rank one with a single prime, where the building is indeed a tree and the action on in gives such a graph of groups.

\subsection{Hausdorff Equivalences among Subgroups of $\glnr$} \label{sec:HausdorffGLnR}

 To make Theorem \ref{thm:graphofzn} more useful, we need to classify subgroups
 of $GL_n\R$ up to finite Hausdorff distance.  This question seems
 interesting in its own right as well.  Clearly one may restrict to closed
 subgroups.  One has the results \cite{Witte} for connected unimodular subgroups generalizing the results on one parameter subgroups from \cite{benerdete} ( both papers interested in questions about foliations with nothing "coarse" about them).  In the next section we work
 out the complete answer for $n=2$.  Here we must be content with a few
 general observations and a conjecture.
 
 \begin{lemma}\label{prod} Let $G$ be a Lie group, and let $K$ be a closed
 subgroup of
 $G \times \R$.  Either there is a closed subgroup $H$ in $G$ so that $K$
 is the graph of a continuous homomorphism $H \to \R$ or $K$ is at finite
 Hausdorff distance from a subgroup of the form $G' \times \R$ for $G'$ a
 subgroup of $G$.
 \end{lemma}
 
 \begin{proof}
  
 Since the normalizer of $K$ is closed in $G \times \R$, and has the form
 $G' \times \R$, we can reduce to the case of $K$ normal in $G \times \R$. 
 
 If $K$ intersects $\R$ non-trivially, then the subgroup of $K$ generated
 by that intersection and the intersection of $K$ with $G$ is cocompact in
 $K$, at hence at finite Hausdorff distance from it, and is clearly
 either of the form $H \times \R$ or the Hausdorff equivalent $H \times
 \Z$.
 
 If $K$ intersects $\R$ trivially, then $K$ is the graph of a homomorphism
 to $\R$ of a subgroup $K'$ of $G$.  The kernel of this homomorphism
 $K''$, which is $K$ intersect $G$, is a closed normal subgroup of $G$. 
 $K$ is then the preimage of a closed subgroup of $(G/K'') \times \R$,
 which is the graph of an injective homomorphism $K'/K'' \to \R$, which
 is clearly abelian.  Thus the lemma reduces to the case of $G$ abelian,
 which is straight forward.  Note, however, that both cases of the
 conclusion of the lemma can occur, even when $K$ intersects $\R$ 
 trivially.
 
 \end{proof}
 
 Thus, to understand all subgroups of $GL_n \R$ up to Hausdorff
 equivalence, one needs to understand the closed subgroups of $SL_n \R$
 and their homomorphisms to $\R$.   The problem of understanding subgroups
 of $SL_n(\R)$ up to Hausdorff equivalence has a nice geometric
 interpretation:  Let $X=SL_n\R / SO_n\R$.  Two subgroups of $SL_n \R$ are
 at finite Hausdorff distance if and only if their orbits in $X$ are at
 finite Hausdorff distance.  Since $X$ is a nice nonpositively curved
 space, one ought to be able to approach the question geometrically,
 perhaps using some appropriate boundary of $X$.  
 
 One case where one can use the symmetric space technology is when the
 subgroup is a non-uniform lattice.
   
 \begin{lemma}\label{ratner} If $\Gamma \in G$ is a
 non-uniform lattice then any group at finite Hausdorff distance from
 $\Gamma$ is commensurable to it.
 \end{lemma}
 
 \begin{proof}
 
 Let the lattices be $\Gamma$ and $\Gamma'$.  The fact that $\Gamma$ is at
 finite Hausdorff distance from $\Gamma'$ implies that the $\Gamma'$ orbit
 of the identity coset of $X/\Gamma$, is bounded.  By the
 generalization of Ratner's theorem in \cite{shah}, this implies that
 there is a closed subgroup of $G$ containing $\Gamma'$ with
 compact orbit of the identity coset in $X/\Gamma$.  Since the only closed
 non-discrete subgroup which contains a non-uniform lattice in all of
 $SL_n(\R)$, this subgroup must be discrete.  This implies that its
 orbit, and hence that of $\Gamma'$ in $X/\Gamma$ is finite, which
 implies that $\Gamma$ and $\Gamma'$ are commensurable. 
 
 \end{proof}
 
\begin{remark}  At least for $\Gamma=GL_n(\Z)$ this can be done without appealing to Ratner (see lemma \ref{ratnerlite})
\end{remark}
 
 \begin{conjecture}\label{zariski} The lemma holds for any discrete Zariski
 dense subgroup which is not cocompact.
 \end{conjecture} 
 
 If true, this should essentially allow one to give a complete
 classification when combined with \cite{Witte}.
 
 A special case of relevance to this paper is free subgroups of
 $GL_n(\Z)$, as these control the groups in the proper holonomy case,
 both up to quasi-isometry and commensurability.    Even here the conjecture seems to be wide open.
 
 \subsection{The case $n=2$}
 
 We can give a complete quasi-isometry classification of graphs of
 $\Z^2$s.  To start, as in the previous section, we need the Hausdorff
 classification of subgroups.
 
 Let $H$ be a closed subgroup of $SL_2 (\R)$.  
 
 If $H$ has dimension at least 2, then $H$ acts cocompactly on the
 hyperbolic plane, and therefore is at finite Hausdorff distance from
 $SL_2 \R$.  If $H$ is one dimensional, then up to conjugacy, there are
 three possibilities for the identity component: elliptic, hyperbolic, or
 parabolic one parameter subgroups.  If is easy to see that, in any of the
 three cases, if $H$ has infinitely many components then it is cocompact
 and therefore again at finite Hausdorff distance from $SL_2 \R$.  The
 hyperbolic and parabolic one parameter subgroups are distinct, and the
 elliptic case is at finite Hausdorff distance from the trivial group.
 
 Thus we are reduced to the discrete case.  Since we are concerned only
 with finitely generated groups, and we are in dimension two, these are
 geometrically finite Fuchsian groups.   The limit set in the circle
 , up to the action of $PSL_2 \R$, is an invariant of the bounded
 Hausdorff class.  If the limit set is finite then the subgroup is at
 bounded Hausdorff distance from one of the one parameter subgroups
 discussed above.  Fuchsian groups with infinite limit set are of
 two kinds: the limit set is the entire circle, or is a Cantor set.  In
 either case, because $H$ is finitely generated, it is acts with
 cofinite volume the hull of its limit set.
 
 For groups with limit set a Cantor set $C$, the subgroup $H'$ of
 $PSL_2(\R)$  which preserves $C$ is discrete.  As $H$ is
 contained in $H'$, and acts with cofinite volume on the convex hull of
 $C$, $H$ has finite index in $H'$.  Thus the limit set is a complete
 invariant of the bounded Hausdorff class, and any two groups in this
 Hausdorff class are commensurable.
 
 Fuchsian groups with limit set $S^1$ are either cocompact or cofinite
 volume.  The cocompact groups, discrete or otherwise, form a single
 quasi-isometry class.   The cofinite volume case is handled by lemma
 \ref{ratner}.
 
 By lemma \ref{prod} a closed subgroup of $GL_2(\R)$ is either the graph
 of a homomorphism from a closed subgroup of $SL_2(\R)$ to $\R$ or
 Hausdorff equivalent to a product of a closed subgroup of $SL_2(\R)$ and
 $\R$.
 
 We therefore need to know, among the Hausdorff classes of closed subgroups
 of $SL_2(\R)$, which have representatives with homomorphisms to $\R$, and
 what those are.
 
 Each one parameter subgroup has a one dimensional family of homomorphisms
 to $\R$, and any Hausdorff equivalent group has at most this family.
 
 The discrete groups with Cantor limit set have many homomorphisms to
 $\R$.  Graphs of two such homomorphisms are Hausdorff close only if the 
 groups in $SL_2(\R)$ are commensurable and the homomorphisms agree on
 the common finite index subgroup.  Thus the graphs are also commensurable.
 
 Among the groups Hausdorff equivalent to $SL_2(\R)$, the only two kinds
 with maps to $\R$ are the upper triangular group, which has a one
 dimensional family of maps, and the cocompact discrete groups, which
 have many.  Thus we need to know when the graphs of such homomorphisms are
 Hausdorff close.
 
 \begin{lemma} Let $G$ and $G'$ be discrete cocompact subgroups of
 $SL_2(\R)$, and $f:G \to \R$ and $f':G' \to \R$ be non-trivial
 homomorphisms to $\R$.  If the graphs of $f$ and $f'$ are Hausdorff
 conjugate in $GL_2(\R)$ then they are commensurable.
 \end{lemma}
 \begin{proof}
 
 Assume $G$ and $G'$ are not commensurable.
 
 The graphs of $f$ and $f'$ are Hausdorff close if and only if for any $r$
 there is an $s$ so that for all $g \in G$ and $g' \in G'$ with
 $d(g,g')<r$ one has $|f(g)-f'(g')|<s$.  This implies that there is a
 continuous function $F:SL_2(\R) \to \R$ which is at bounded distance from
 both $f$ and $f'$.
 
 For any element $x=gg' \in GG'$ we have that, for any $y$, $$|F(xy)-F(y)|
 = |F(g(g'y)) - F(g'y) + F(g'y) - F(y)|$$
 $$ \leq |F(g(g'y))-F(g'y)| + |F(g'y)-F(y)|$$ 
 
 Thus $|F(xy)-F(y)|\leq M$ for some $M$ which is independent of $x$ and
 $y$.  Since the set of $x$ with this property is closed, and $GG'$ is
 dense, this is true for all $x$ and $y$.
 
 This implies $F$ is bounded, since the $SL_2(\R)$ action on itself by
 left translation is transitive.  This implies both $f$ and $f'$ are
 bounded, and hence are both identically zero.
 
 \end{proof}
 
 In the same way one checks that the graphs of the homomorphisms from the
 upper triangular subgroup are not Hausdorff close to those from cocompact
 lattices.
 
 From this one can write down a complete (and long and tedious) list of
 the Hausdorff conjugacy classes of subgroups of $GL_2(\R)$ and thereby a list of the quasi-isometry types of graphs of $\Z^2$s.

\subsection{Commensurability of Free-by-$\Z^n$ groups}

To compare with the quasi-isometry classification we also want to determine the commensurability classification of graphs of $\Z^n$s.    Even for $n=1$ this runs into difficulties in general.    However, in the semi-direct product case this can be done.

 \begin{theorem}\label{commnielsen} Let $G$ be a finitely generated group,
 and let
 $f:F
 \to G$ and $f':F' \to G$ be homomorphisms from non-abelian free groups
 with finite index images.  Unless one of $f$ and $f'$ is injective while
 the other is not, there are finite index subgroups $F_1 \subset F$ and
 $F'_1
 \subset F'$ and an isomorphism $\phi:F_1 \to F'_1$ which commutes with the
 restrictions of $f$ and $f'$.
 \end{theorem}
 
 \begin{proof}
 
 We may assume, by passing to finite index subgroups, that $f$ and $f'$
 are both surjective.  The case of both maps injective is trivial, so we
 also assume neither map is injective.
 
 As $F$ is free, there is then a lift of $f$ to a map $\hat{f}:F \to F'$
 such that $f=f'\hat{f}$.  Let $F_0$ be the kernel of $\hat{f}$, and $F'_0$
 the image.
 
 By a theorem of Whitehead, $F$ splits as a free product $F_0 * K$ with
 $\hat{f}$ taking $K$ isomorphically to $F'_0$.  
 
 Since $F'_0$ is the image of a finitely generated group, it is finitely
 generated.  Now, by a theorem of Hall, there is a finite
 index subgroup $F'_1$ of $F'$ containing $F'_0$ which splits as a free
 product of $F'_0 * H$  for some finitely generated free $H$.  Since
 $f'(F'_0)=G$, we can take $H$ contained in the kernel of $f'$.
 
 We have an isomorphism of $F$ with $F'_0 * K$ which commutes with
 the maps $f$ and $f'$ to $G$.  Likewise, there is a finite index
 subgroup of $F'$ with an isomorphism to $F'_0 *  H$ which commutes with
 the maps to $G$.
 
 Because we assume that the maps from $F$ and $F'$ to $G$ have kernel, we
 can pass to further finite index subgroups to guarantee that $H$ and $K$
 are non-trivial.  At this point things are reduced to:
 
 \begin{lemma}  Let $A$, $B$, and $C$ be non-trivial free groups.  There
 are finite index subgroups $D$ of $A*C$ and $D'$ of $B*C$ with
 isomorphisms that commute with the maps to $C$. 
 \end{lemma}
 
 By \cite{whitehead} there is, up to isomorphism, only one surjection
 between free groups of any given ranks.  Thus the lemma is equivalent to
 showing that there are finite index subgroups $D$ and $D'$ of the same
 ranks which both surject to $C$.  It is an easy calculation to check that
 this occurs with $D$ and $D'$ the kernels of maps to finite groups which
 factor through $A$ and $B$ respectively.
 
 \end{proof}
 
 \begin{corollary}\label{semihol} If $G=H \rtimes
 _f F$ and $G'=H \rtimes
 _f F'$
 then $G$ and $G'$ are fiber respecting commensurable iff
 either both or neither of $f$ and $f'$ are injective, and $f(F)$ and
 $f'(F')$ are commensurable in $Comm(H)$.
 \end{corollary}
 
 \begin{proof}
 The "only if" direction is trivial.  Assume that $f$ and $f'$ are as in
 the corollary.  By the lemma there are finite index subgroups $F_1$ of $F$
 and $F_1'$ of $F'$, an isomorphism, $T$, between them, and an element
 $x$ of $Comm(H)$ so that 
 $$x^{-1} f(w) x = f'( Tw)$$
 for all $w \in F_1$.
 
 After realizing $x$ by an isomorphism between two finite index
 subgroups of $H_1$ and $H'_1$, this is exactly an isomorphism
 between $H_1 \rtimes
 _f F_1$ and $H'_1 \rtimes
 _{f'} F'_1$.     
 \end{proof}
 
 As with quasi-isometries, when $H$ has Poincare duality (and, in particular, for $H=\Z^n$) all commensurators are fiber preserving (unless $F=\Z$).
 
 \begin{corollary}\label{abcomm} The groups of
 the form $\Z^n \rtimes
 _f F$ fall into two classes depending on whether the
 action of $F$ is faithful.   These classes are closed under
 commensurability, and within each class the commensurability
 classes are in one-to-one correspondence, given by the image of $f$, with
 commensurability classes within $GL_n(\Q)$ of finitely generated subgroups
 of $GL_n(\Z)$.   
 \end{corollary}
 
 It is instructive to compare this with Theorem \ref{thm:graphofzn}.   One difference is that the parabolic case cannot occur here.   Another is that commensurability in $GL_n(\Q)$ here changes to bounded Hausdorff equivalence in $\glnr$ in that theorem.   Since the only holonomy groups which can occur here are subgroups of $\glnz$ the difference may not be that big - certainly for $n=2$ there is none.

\medskip
\noindent
Kevin Whyte\\
Dept. of Mathematics\\
University of Illinois at Chicago\\
Chicago, Il 60607\\
E-mail: kwhyte@math.uic.edu

\end{document}